\documentclass[12pt,reqno]{amsart}
 \usepackage[utf8]{inputenc}
 \usepackage[T1]{fontenc}% Pour pouvoir taper les accent directement et non pas passer par
 \usepackage[usenames, dvipsnames]{color}
 \usepackage{ulem}
 
 \usepackage{dsfont, amsfonts, amsmath, amssymb,amscd, stmaryrd, latexsym, amsthm, dsfont}
 \usepackage[frenchb,english]{babel}
 \usepackage{enumerate}
 \usepackage{longtable}
 \usepackage{geometry}
 \usepackage{float}
 \usepackage{tikz}
 \usetikzlibrary{shapes,arrows}
 \usepackage{hyperref}
 \hypersetup{
 	colorlinks=true,
 	urlcolor=blue,
 	citecolor=blue}

 \newtheorem{theorem}{Theorem}[section]
 \newtheorem{lm}[theorem]{Lemma}
 \newtheorem{prop}[theorem]{Proposition}
 \newtheorem{corollary}[theorem]{Corollary}

 \theoremstyle{remark}
 \newtheorem{rema}[theorem]{\bf Remark}
 
 \newtheorem{exam}[theorem]{\textbf{Example}}
 \newtheorem{exams}[theorem]{\textbf{Example}}

 \def\NN{\mathds{N}}
 
 \def\QQ{\mathbb{Q}}
 
 \def\ZZ{\mathbb{Z}}
 
\begin{document}
  	
  \selectlanguage{english}
  \title[On the rank  of the $2$-class group...]{On the rank  of the $2$-class group of some imaginary triquadratic number  fields}
  %premier auteur
  \author[A. Azizi]{Abdelmalek Azizi}
  \address{Abdelmalek Azizi: Mohammed First University, Mathematics Department, Sciences Faculty, Oujda, Morocco }
  \email{abdelmalekazizi@yahoo.fr}
  % deuxième auteur
  \author[M. M. Chems-Eddin]{Mohamed Mahmoud CHEMS-EDDIN}
  \address{Mohamed Mahmoud CHEMS-EDDIN: Mohammed First University, Mathematics Department, Sciences Faculty, Oujda, Morocco }
  \email{2m.chemseddin@gmail.com}
  
  % troisième auteur
  \author[A. Zekhnini]{Abdelkader Zekhnini}
  \address{Abdelkader Zekhnini: Mohammed First University, Mathematics Department, Pluridisciplinary faculty, Nador, Morocco}
  \email{zekha1@yahoo.fr}
  

  \begin{abstract}
  	Let $d$ be  an odd square-free integer and $\zeta_8$  a primitive $8$-th root of unity. The purpose of  this paper is to investigate  the rank of the $2$-class group of the fields  $L_d=\mathbb{Q}(\zeta_8,\sqrt{d})$.
  \end{abstract}
  
  \selectlanguage{english}
  
  \maketitle
  
  \section{Introduction}
  
  Let $k$ be a number field and $p$ a prime integer. Let $\mathrm{Cl}_p(k)$ denote the  $p$-class group of $k$, that is the $p$-Sylow subgroup  of its ideal class group $\mathrm{Cl}(k)$ in the wide sense. Class groups of  number fields have been studied for a long time, and
  there are many very interesting (and very difficult) problems concerning their behavior.
  An interesting invariant is the $p$-rank of $\mathrm{Cl}_p(k)$, i.e. the dimension   of  $\mathrm{Cl}(k)/\mathrm{Cl}(k)^p$ as a vector space over the field with $p$ elements $\mathbb{F}_p$.
   % A particular quantity of interest is  the rank of $\mathrm{Cl}_p(k)$ %which can be defined as follows.
 % It is known, by the Kronecker decomposition theorem, that any  finite %abelian  group   is   the direct product of cyclic groups. Then the rank %of $\mathrm{Cl}_p(k)$  is the number of cyclic $p$-groups appearing in %the decomposition of $\mathrm{Cl}(k)$, i.e,  the dimension of the %elementary abelian $p$-group  $\mathrm{Cl}(k)/\mathrm{Cl}(k)^p$ viewed %as a vector space over the prime field with p elements $\mathbb{F}_p$.
    In this work, we  investigate the $2$-rank  of the class groups of some imaginary triquadratic number  fields. To the best of our knowledge, there is no study in this setting and it
  would be an interesting task to develop such investigations. Note that the odd part of the class group of a number field is much better understood, since it is known to be isomorphic to the direct product of the odd parts of the class groups of its quadratic subfields (see \cite{lemmermeyer1994kuroda}).
  
  In the literature, there are many studies  that investigated  the $2$-rank of the class group of a number field $k$,  let us quote some.  For a quadratic field $k$, using   Gauss's genus  theory, one can easily deduce the rank of $\mathrm{Cl}_2(k)$.  For biquadratic fields,  the authors of \cite{azizi99, Be05} determined all positive  integers $d$  such that  $\mathrm{Cl}_2(k)$ is  isomorphic to  $\ZZ/2\ZZ \times \ZZ/2\ZZ$, where $k=\QQ(\sqrt d, \sqrt{-\ell})$ and   $\ell=1, 2$.  The  papers \cite{azizi20042, AM}  investigated the rank of $\mathrm{Cl}_2(k)$ for the fields
  $k=\QQ(\sqrt d, \sqrt{m})$, where $m$ is a prime and  $d$ a positive   square-free integer.  In the same direction,
  \cite{azizitaous(24)(222)}  classified  all fields $k=\QQ(\sqrt d, i)$ such that $\mathrm{Cl}_2(k)$ is isomorphic to  $\ZZ/2\ZZ \times \ZZ/4\ZZ$ or  $(\ZZ/2\ZZ)^3 $.
  In \cite{Parry77, brown19782}, E. Brown and C. Parry determined some imaginary quartic cyclic number  fields $k$ such that
  the rank of $\mathrm{Cl}_2(k)$ is at most $3$.  Finally,  \cite{mccall1995imaginary}  determined all  imaginary biquadratic fields whose $2$-class group  is cyclic.
  
  Let $d$ be  an odd  square-free integer and $\zeta_8$  a primitive $8$-th root of unity. In the present work, we are interested in  the rank of the $2$-class group of the  imaginary triquadratic number   fields  $L_{d}:=\mathbb{Q}(i,\sqrt{2},\sqrt{d})= \mathbb{Q}(\zeta_{8},\sqrt{d})$.
  Methods and techniques  used   are based on genus theory,  ambiguous class number formula,  the properties of the norm residue  symbol, units of some number fields  and  some other results of algebraic number theory.
  
  The structure of this paper is the following. In \S\ \ref{1}, we recall the ambiguous class number formula and the relations   allowing to calculate the rank of the $2$-class group of $L_d$. Next, we compute the number of prime ideals  of $K=\mathbb{Q}(\zeta_{8}) $ ramified in  $L_{d}=K(\sqrt{d})$. Thereafter, in \S\ \ref{6}, we recall the definition and some properties of the quadratic  norm residue symbol. In \S\  \ref{secion sur le rang de Ld}, we investigate the rank of $\mathrm{Cl}_2(L_d)$ according to  classes $\pmod 8$ in which  the prime divisors of $d$ lie. As  applications, in \S\ \ref{section sur les applications} we  determine the integers $d$ such that  $\mathrm{Cl}_2(L_d)$ is trivial, cyclic or of rank $2$, and we thus deduce the integers $d$ satisfying $\mathrm{Cl}_2(L_d)\simeq (\ZZ/2\ZZ)^2$, we shall say  $\mathrm{Cl}_2(L_d)$ is of type $(2, 2)$. We end this section by giving  the $2$-part of the class number of $L_d$ in terms of   those of its  biquadratic subfields.% Finally, and as an appendix, in \S\ \ref{3}, we  present the polynomial generating the field $L_d$.
  \section*{Notations}
  Let $k$ be a number field. Throughout this paper, we use  the following notations.
  \begin{enumerate}[\rm$\bullet$]
  	\item $d$:  An odd square-free integer,
  	\item $\zeta_{n}$:    An $n$-th primitive root of unity,
  	\item $K$   $:=\mathbb{Q}(\zeta_{8})$,
  	\item$L_{d}$  $:=K(\sqrt{d})$,
  	\item $\delta_{k}$:   The  discriminant of $k$,
  	\item $\mathcal O_{k}$:   The ring of integers of  $k$,
  	 	\item $\mathrm{Cl}(k)$:   The class group of  $k$,
  	\item $\mathrm{Cl}_2(k)$:   The $2$-class group of  $k$,
  	\item $h(k)$:   The class number of $k$,
  	\item $h_2(k)$:   The $2$-class number of $k$,
  	\item $ h_2(m)$:   The $2$-class number of a quadratic  field $\mathbb{Q}(\sqrt m)$,
  	\item $N$:   The  norm map for the extension $L_{d}/K$,
  	\item $E_{k}$:  The unit group of $\mathcal O_{k}$, %The group of units in the ring of integers of $k$,
  	\item $e_{d}$:   The integer defined by $(E_{K}:E_{K}\cap N(L_{d}))=2^{ e_{d}}$,
  	\item $\left( \frac{\alpha,\, d}{\mathfrak p}\right)$:   The quadratic norm residue symbol 	over $K$,
  	\item $\left(\dfrac{\cdot}{\cdot}\right)_4$: The biquadratic residue symbol,
  	\item $\overline{\alpha}$:   The coset of $\alpha$ in $E_{K}/(E_{K}\cap N(L_{d}))$,
  	\item $r_2(d)$:   The rank of the $2$-class group of $L_d$,
  	\item $\varepsilon_m$:   The fundamental unit of $\mathbb{Q}(\sqrt{m})$, where $m$ is a positive square-free integer,
  	\item $W_{k}$:   The set of roots of unity contained in $k$,
  	\item $\omega_{k}$:   The cardinality of $W_{k}$,
  	\item $k^{+}$:   The maximal real subfield of  $k$,
  	\item $Q_{k}$:   The  Hasse's index, that is  $(E_{k}:W_{k}E_{k^{+}})$, if $k/k^+$ is CM,
  	\item $Q(k'/k)$:   The unit index of a  biquadratic extension  $k'/k$,
  	\item $q(k):=(E_{k}: \prod_{i}E_{k_i})$, where $k_i$ are  the  quadratic subfields	of a multiquadratic field $k$.	
  \end{enumerate}
  \section{Preliminaries}\label{1}
  Let $K:=\mathbb{Q}(\zeta_{8}) $ and $L_{d}:=K(\sqrt{d})$, where $d$ is an odd square-free  integer. Note that
  $\mathcal{O}_K$ is a principal ideal domain. Denote by $Am(L_{d}/K)$ the group of ambiguous ideal  classes
  of $L_{d}/K$, that are classes of $\mathrm{Cl}(L_{d})$ fixed under any element
  of $\mathrm{Gal}(L_{d}/K)$, and by $Am_2(L_{d}/K)$ its 2-Sylow subgroup. In fact, $Am_2(L_{d}/K)=\{\mathfrak c\in \mathrm{Cl}({L_{d}}): \mathfrak{c}^2=1\}$ is an elementary  $2$-group of order
  $2^{r_2(d)}.$
  
  It is  well known, according to the ambiguous class number formula for a cyclic extension of prime degree (cf. \cite{gras1973classes}), that
  \begin{eqnarray*}
  	|Am(L_{d}/K)|=h(K)\frac{2^{t_{d}-1}}{( E_{K}:E_{K}\cap N(L_{d}))}.
  \end{eqnarray*}
  It follows that
  \begin{eqnarray}
  |Am_2(L_{d}/K)|=\frac{2^{t_{d}-1}}{( E_{K}:E_{K}\cap N(L_{d}))}=2^{t_{d}-1-{e_{d}}},
  \end{eqnarray}
  where $t_{d}$   is the number of finite  and infinite primes  of ${K}$ which ramify in $ L_{d}/K$ and ${e_{d}}$ is defined by $(E_{K}:E_{K}\cap N(L_{d}))=2^{ e_{d}}$. So the $2$-rank of $\mathrm{Cl}(L_{d})$ verifies the relation:
  \begin{eqnarray}
  r_2(d)=t_{d}-1-{e_{d}}. \label{egalite du 2-rang   =t-1-e}
  \end{eqnarray}

  Let us now  determine the ideals of $K$ that ramify in $L_d$. 	
    Let $\delta_{L_{d}/K}$ denote the generator of the relative discriminant of the extension $L_{d}/K$. We have:

  \begin{prop}\label{ prop the dicriminant}
  	Let $d$ be an odd  square-free integer. Then the relative discriminant of  $L_{d}/K$ is generated by $\delta_{L_{d}/K}= d$.
  \end{prop}

  \begin{proof}
  	Assume that $d\equiv 1\pmod 4$. Thus $\delta_{\mathbb{Q}(\sqrt{d})}=d$,  $\delta_{K}=2^{8}$ and $\mathcal{O}_{\mathbb{Q}(\sqrt{d})}=\mathbb{Z}[\frac{1+\sqrt{d}}{2}]$.
  	Since $\delta_{\mathbb{Q}(\sqrt{d})}$ and $ \delta_{K}$ are coprime, then $\mathcal{O}_{L_{d}}=\mathcal{O}_{K}[\frac{1+\sqrt{d}}{2}]$. Hence    $\delta_{L_{d}/K}= \mathrm{disc}_{L_{d}/K}\left(1,\frac{1+\sqrt{d}}{2}\right)=d$.
  		If  $d\equiv 3\pmod 4$, then $-d\equiv 1\pmod 4$. As  $L_{d}=L_{-d}$, the previous case completes the proof.
  \end{proof}
  \noindent The following result is easily deduced.
  \begin{corollary}
  	Let $d$ be an odd square-free integer.
  	\begin{enumerate}[\rm1.]
  		\item The discriminant of $L_d$ is   $\delta_{L_{d}}=2^{16}\cdot d^{4}$, and thus the primes that ramify in $L_{d}$ are exactly $2$ and the prime divisors of $d$.
  		\item The ring of integers of $L_{d}$ is given by	
  		$$	\mathcal{O}_{L_{d}}=\left\{  \begin{array}{ccc}
  		\mathbb{Z}[\zeta_{8},\frac{1+\sqrt{d}}{2}]& \text{ if}& d\equiv 1\pmod 4,\\
  		\mathbb{Z}[\zeta_{8},\frac{1+\sqrt{-d}}{2}]& \text{ if}& d\equiv 3\pmod 4.
  		\end{array}
  		\right.$$	
  	\end{enumerate}
  	%(See $7Q$ of \cite{ribenboim1972algebraic}).
  \end{corollary}
  \noindent Next, we compute the number of prime ideals of $K$ that ramify in $L_d$.
  \begin{theorem}\label{prop number of ramified primes}
  	Let 	$d$ be an odd  square-free integer. Then  the number of prime ideals of  $K$ that ramify in $L_{d}$ is $2(q+r)$, where $r$ is the number of  prime integers  dividing  $d$ and $q$ is the number of those which are congruent to $ 1\pmod 8$.
  \end{theorem}
  \begin{proof}
  	Using 	the theorem of the cyclotomic reciprocity law (\cite[ Theorem 2.13]{washington1997introduction}), one can  verify what follows:
  	
  	\hspace{0.5cm}\begin{minipage}{13cm}
  		\begin{enumerate}[\rm $\bullet$]
  			\item If $p$ is a prime integer  such that $p\equiv 1\pmod 8$, then there are exactly $4$ prime ideals of $K$ lying over $p$.
  			\item If  $p\not\equiv 1\pmod 8$, then there are exactly $2$ primes of $K$ above $p$.
  		\end{enumerate}
  	\end{minipage}\\
  	By Proposition \ref{ prop the dicriminant}, the prime ideals of $K$ that ramify in $L_{d}$ are exactly the divisors of $d$ in $\mathcal{O}_K$. So the number of those primes is   $4q+2(r-q)=2(q+r)$ as desired.
  \end{proof}
  
  Since $K$ is a biquadratic field,   $e_d\in \{0,1,2\}$. The  next corollary follows directly from the previous theorem and   the formula \eqref{egalite du 2-rang   =t-1-e}.
  \begin{corollary}
  	Let $d$ be an odd  square-free integer. Then  $h(L_{d})$ is even in the following two cases:
  	\begin{enumerate}[\rm1.]
  		\item  $|d|$ is composite.
  		\item $|d|$ is a prime  congruent to $1\pmod 8$.
  		
  	\end{enumerate}
  \end{corollary}
  \noindent The next lemma provides the unit group of $K$.
  \begin{lm}\label{lm unite zeta8}
  	The unit group  of $K$ is given by
  	$E_{K}=\langle\zeta_8,\varepsilon_2\rangle.$
  \end{lm}
  \begin{proof}
  	The field $K=\mathbb{Q}(\zeta_8)=\mathbb{Q}(i,\sqrt{2})$ is a Galois extension of $\mathbb{Q}$ whose Galois group is an elementary $2$-group of order $4$.
  	We have, $\mathbb{Q}(\sqrt{2})$ is the  real quadratic subfield of $K$ with fundamental unit $\varepsilon_2=1+\sqrt{2}$, which is also a fundamental unit of $K$. Thus by   \cite[Theorem 42]{frohlich1993}, the Hasse's unit index is equal to $1$, i.e.,
  	$ (E_{K}:E_{\mathbb{Q}(\sqrt{2})}W_{K}) =1$, where $W_{K}$ is the set of roots of unity contained in $K$.
  \end{proof}
  \section{Quadratic norm residue symbol}\label{6}
  To compute the index $e_d$ appearing in the formula (\ref{egalite du 2-rang   =t-1-e}), we will use  the quadratic norm residue symbol, so we have to recall its  definition  and some of its  properties (cf.  \cite[Chapter II, Theorem 3.1.3]{grasbook}).
  Let $k$ be a number field and $\beta\in k^*$  a square-free element in $k$. Let $f$ be the conductor of $k(\sqrt{\beta})/k$. For any prime
  $\mathfrak{p}$ of $k$ (finite or infinite), we denote by $f_\mathfrak{p}$  the largest power of  $\mathfrak{p}$ dividing  $f$. Let $\alpha\in k^*$, according to the approximation theorem there exists $\alpha_0\in k$ such that
  \begin{center}
  	\begin{tabular}{ccc}
  		$\alpha_0\equiv \alpha   \pmod{f_\mathfrak{p}}$& and & $ \alpha_0\equiv 1  \pmod{ \frac{f}{f_\mathfrak{p}}}$.
  	\end{tabular}
  \end{center}
  If $(\alpha_0)=\mathfrak{p}^n\mathfrak{P}$ with $n\in  \mathbb{Z}$ and $(\mathfrak{P},\mathfrak{p})=1$ ($n=0$ if $\mathfrak{p}$ is infinite), set
  
  $$\left(\frac{\alpha,k(\sqrt{\beta})}{\mathfrak{p}}\right)=\left(\frac{k(\sqrt{\beta})}{\mathfrak{P}}\right),$$
  where $\left(\frac{k(\sqrt{\beta})}{\mathfrak{P}}\right)$ is the Artin map applied to $\mathfrak{P}$. For $\alpha\in k^*$ and a prime (finite or infinite) $\mathfrak{p}$ of $k$, the quadratic norm residue symbol
  is defined by
  $$\left(\frac{\alpha,\beta}{\mathfrak{p}}\right)=\frac{\left(\frac{\alpha,k(\sqrt{\beta})}{\mathfrak{p}}\right)(\sqrt{\beta})}{\sqrt{\beta}}\in\{\pm 1\}.$$
  
  If the prime $\mathfrak{p}$ is unramified in $k(\sqrt{\beta})/k$, we set
  
  $$\left(\frac{\beta}{\mathfrak{p}}\right)=\frac{\left(\frac{k(\sqrt{\beta})}{\mathfrak{p}}\right)(\sqrt{\beta})}{\sqrt{\beta}}\in\{\pm 1\}.$$
  
  Note that the norm residue symbol may be defined more generally for an extension $k(\sqrt[m]{\beta})/k$, where $m\in\mathds{N}^*$,  $k$ is a number field containing the $m$-th root of unity and $\beta \in k^*$.
  The quadratic norm residue symbol  verifies the following properties that we shall use later.
  
  \begin{enumerate}[\rm 1.]
  	\item $\left(\frac{\alpha_1\alpha_2,\;\beta}{\mathfrak{p}}\right)= \left(\frac{\alpha_1,\;\beta}{\mathfrak{p}}\right)
  	\left(\frac{\alpha_2,\;\beta}{\mathfrak{p}}\right).$
  	\item $\left(\frac{\alpha,\;\beta}{\mathfrak{p}}\right)=\left(\frac{\beta,\; \alpha}{\mathfrak{p}}\right).$
  	\item If $\mathfrak{p}$ is unramified in $k(\sqrt{\beta})/k$ and appears with exponent $e$ in the decomposition of  $(\alpha)$, then
  	$\displaystyle\left(\frac{\alpha,\beta}{\mathfrak{p}}\right)=\left(\frac{\beta}{\mathfrak{p}}\right)^e.$
  	\item  If $\mathfrak{p}$ is unramified in $k(\sqrt{\beta})/k$ and does not appear in  the decomposition of  $(\alpha)$, then $\displaystyle\left(\frac{\alpha,\beta}{\mathfrak{p}}\right)=1. $
  	\item $\prod_{\mathfrak{p}\in Pl}\left(\frac{\alpha,\;\beta}{\mathfrak{p}}\right)=1$, where $Pl$ is the set of all finite and infinite primes of $k$.
  	\item Let $k_1$ be a finite extension of $k$, $\alpha \in k_1^*$ and $\beta\in k^*$. Denote by $\mathfrak{p}$ a prime ideal of $k$ and by $\mathfrak{B}$ a prime ideal of $k_1$ above $\mathfrak{p}$. Thus
  	$$\prod_{\mathfrak{B}|\mathfrak{p}}\left(\frac{\alpha,\beta}{\mathfrak{B}}\right)=\left(\frac{N_{k_1/k}(\alpha),\beta}{\mathfrak{p}}\right).$$
  \end{enumerate}
  
  \section{The rank of the $2$-class group of $L_d$}\label{secion sur le rang de Ld}
  In the present section we compute, $r_2(d)$, the rank of the $2$-class group of $L_d$.
  Since $L_{d}=L_{-d}=L_{2d}$, then without loss of generality, we  suppose that $d$ is an odd positive square-free integer. We shall compute $r_2(d)$ distinguishing two cases  according to the  classes $\pmod 8$ in which the prime  divisors of $d$  lie. Let us start with the following results.
  
  \begin{lm}[\cite{azizi20042}]
  	Let $k$ be a number field, $d$ a positive square-free integer and $\alpha \in k^*$ such that the ideal   $ \alpha\mathcal{O}_k$   is the norm of a fractional ideal of $k(\sqrt{d})$.
  	Then, $\alpha$ is norm in $k(\sqrt{d})/k$ if and only if $\displaystyle\left(\frac{\alpha,d}{\mathfrak{p}}\right)=1$ for  all primes $\mathfrak{p}$ of $k$   ramified in  $k(\sqrt{d})$.
  \end{lm}
  %	After  the above  lemma, Proposition \ref{ prop the dicriminant} and Lemma \ref{lm unite zeta8} to compute $e_d$, it suffices to compute the quadratic norm residue symbols  $\left(\frac{\zeta_{8},d}{\mathfrak{p}}\right)$ and $\left(\frac{\varepsilon_2,d}{\mathfrak{p}}\right)$
  %	only  for the  prime ideals $\mathfrak{p}$ of $K$ dividing $d$. We have:

  Thus, to compute $e_d$, it suffices to compute the quadratic norm residue symbols  $\left(\frac{\zeta_{8},\,d}{\mathfrak{p}}\right)$ and $\left(\frac{\varepsilon_2,\,d}{\mathfrak{p}}\right)$
     for the  prime ideals $\mathfrak{p}$ of $K$ dividing $d$ (see   Proposition \ref{ prop the dicriminant} and Lemma \ref{lm unite zeta8} ).
 
The next lemma follows directly from properties 1, 4 and 5 of the
 norm residue symbol.
  \begin{lm}\label{lemma norm residue symbol pp} Let  $d$ be an odd  positive  square-free integer and  $\mathfrak p$ a prime of $K$ dividing $d$. Denote by $\alpha$ a unit of $K$. Then
  	\begin{enumerate}[\rm1.]	
  		\item 	$\displaystyle\left(\frac{\alpha,\,d}{\mathfrak p}\right)= \left( \frac{\alpha,\,p}{\mathfrak p}\right),$ where $p$ is the prime contained in $\mathfrak p$.
  		\item  $\displaystyle\prod_{\mathfrak p | d}	\left( \frac{\alpha,\, d}{\mathfrak p}\right)=1.$	
  	\end{enumerate}	
  \end{lm}
 % \begin{proof}
  %	\begin{enumerate}[\rm1.]	
  %		\item Let $q$ be a prime dividing $d$. We have
  %		$$\begin{array}{ll}
  %		\displaystyle\left( \frac{ \alpha,\,d}{\mathfrak p}\right)=\prod_{q|d} \left( \frac{ \alpha,\,q}{\mathfrak p}\right)=\left( \frac{ \alpha,\,p}{\mathfrak p}\right)\prod_{q|d\; \& \;q\not= p} \left( \frac{ \alpha,\,q}{\mathfrak p}\right).
 % 		\end{array}$$
  %		If  $\mathfrak p$ does not divide $ q$, then by Proposition \ref{ prop the dicriminant}, $\mathfrak p$ is unramified in $L_{q}$.  Thus  $\displaystyle\left( \frac{\alpha,q}{\mathfrak p}\right)=\left( \frac{q}{\mathfrak p}\right)^0=1$.
 % 		\item The second item follows directly  from Proposition \ref{ prop the dicriminant} and \cite[Lemme 3]{azizi20042}.	\end{enumerate}\end{proof}
  
  \subsection{Case 1: The prime divisors of $d$ are in the same coset of $\mathbb{Z}/8\mathbb{Z}$}~\\
  Let $d$ be an odd positive square-free integer such that the primes $p|d$ are  in the same coset $\pmod 8$.
  \begin{lm}\label{lm symbols 1}Let  $p$ be   a prime such that $p\not\equiv 1\pmod 8$ and denote by $\mathfrak p_K$ any prime ideal of $K$ above $p$.
  	\begin{enumerate}[\rm 1.]
  		\item If $p\equiv 3\pmod 8$, then
  		$\displaystyle\left(\frac{ \varepsilon_2,\, p}{\mathfrak p_{ K}}\right)=-1 \text{ and }
  		\left( \frac{\zeta_{8},\, p}{ \mathfrak p_{K}}\right)=-1.$
  		\item If $p\equiv 5\pmod 8$, then
  		$\displaystyle\left( \frac{\zeta_{8},\, p}{\mathfrak p_{K}}\right)=-1   \text{ and }  \left( \frac{ \varepsilon_2,\, p}{\mathfrak p_{K}}\right)=1.$
  		\item If $p\equiv 7\pmod 8$, then
  		$\displaystyle\left( \frac{ \varepsilon_2,\, p}{\mathfrak p_K}\right)=1 \text{ and }
  		\left( \frac{\zeta_{8},\, p}{\mathfrak p_K}\right)=1.$	
  	\end{enumerate}
  \end{lm}
  \begin{proof}	 Suppose that $p$ is congruent to $3\pmod 8$.
  	Let $\mathfrak p_{\mathbb{Q}( i)}$ and $\mathfrak p_{\mathbb{Q}(\sqrt{2})}$ be the prime ideals of $\mathbb{Q}(i)$ and $\mathbb{Q}(\sqrt{2})$ respectively  above $p$. Note that these two prime ideals are totally decomposed in $K$. As
  	$p$ is inert in both $\mathbb{Q}(i)$ and  $\mathbb{Q}(\sqrt{2})$,  then
  	$$\begin{array}{ll}
  	\left(\frac{\zeta_{8},\,p}{ \mathfrak p_K}\right)&=\left( \frac{1+i,\, p}{ \mathfrak p_K}\right)\left( \frac{\frac{1}{\sqrt 2},\, p}{ \mathfrak p_K}\right)=\left( \frac{1+i,\, p}{ \mathfrak p_K}\right)\left( \frac{\sqrt 2,\, p}{ \mathfrak p_K}\right)\\
  	&=\left( \frac{1+i}{ \mathfrak p_K}\right)\left( \frac{\sqrt 2}{ \mathfrak p_K}\right)
  	=\left( \frac{1+i}{ \mathfrak p_{\mathbb{Q}(i)}}\right)\left( \frac{\sqrt 2}{ \mathfrak p_ {\mathbb{Q}(\sqrt 2)}}\right)\\
  	&=- \left( \frac{2}{ p}\right)^2=-1,
  	\end{array}$$
  	$$\text{and }\begin{array}{ll}
  	\left( \frac{1+\sqrt{2},\, p}{  \mathfrak p_K}\right)=\left( \frac{1+\sqrt{2}}{  \mathfrak p_K}\right)=\left( \frac{1+\sqrt{2}}{  \mathfrak p_{\mathbb{Q}(\sqrt{2})}}\right)=\left( \frac{-1}{ p}\right)=-1.
  	\end{array}$$
  	We  similarly   prove the other cases.
  \end{proof}

  \begin{theorem}\label{thm p 3 mod 8}
  	Let  $p>2$ be   a prime such that $p\not\equiv 1\pmod 8$.
  	\begin{enumerate}[\rm 1.]
  		\item If $p\equiv 3$ or $5\pmod 8$, then the  $2$-class group of $L_p$
  		is trivial, i.e., $r_2(p)=0.$
  		\item If $p\equiv 7\pmod 8$, then \;the  $2$-class group of \;$L_p$
  		is   \;cyclic nontrivial, i.e., $r_2(p)=1.$
  	\end{enumerate}
  \end{theorem}
  \begin{proof}
  	By Theorem  \ref{prop number of ramified primes}, we have $r_2(p)= 2-1- e_p=1-e_p$,
  	thus  $e_p\in\{0,1\}$. According to the previous Lemma \ref{lm symbols 1}, we have  $e_p= 1$ if $p\equiv 3$ or $5\pmod 8$, and $e_p= 0$ otherwise. Hence,  the results.
  \end{proof}

  \begin{theorem}\label{thm div con 3}
  	Let $d>2$ be a composite odd   square-free integer. Denote by $r$ the number of distinct primes dividing $d$.
  	\begin{enumerate}[\rm 1.]
  		\item If  all the primes  dividing $d$ are congruent to $3\pmod 8$ $($resp. $5\pmod 8$$)$,   then $r_2(d)=2r-2.$
  		\item If  the primes dividing  $d$ are congruent to $7\pmod 8$, then $r_2(d)=2r-1.$
  	\end{enumerate}
  	
  \end{theorem}
  \begin{proof}Assume that all the primes $p|d$ are congruent to $3\pmod 8$.
  	Let   $\mathfrak p_K$ be a prime ideal of $K$ above $p$.	By Lemmas \ref{lemma norm residue symbol pp} and \ref{lm symbols 1}, we have
  	$$\left(\frac{\zeta_{8},\, d}{\mathfrak p_K}\right)=\left(\frac{\zeta_{8},\, p}{\mathfrak p_K}\right)=-1 \text{ and  }\left(\frac{ \varepsilon_2,\, d}{\mathfrak p_K}\right)=\left(\frac{ \varepsilon_2,\, p}{\mathfrak p_K}\right)=-1,$$	 thus $\left(\frac{\zeta_{8}\varepsilon_2,\, d}{\mathfrak p_K}\right)=1.$
  	It follows that $\overline{ \zeta_{8}}=\overline{\varepsilon_2}$  in $E_{K}/(E_{K}\cap N(L_p))$. We infer that  $E_{K}/(E_{K}\cap N(L_p))=\{\overline{1}, \overline{ \zeta_{8}}\}$  and  $e_d=1$. By Theorem \ref{prop number of ramified primes},  the number of prime ideals of $K$ ramified in $L_{d}$ is $2r$. Hence,
  	$r_2(d)=2r-1-1=2r-2.$
  	
  	The other cases are similarly proved.
  \end{proof}
  
  In what follows, we will compute $r_2(d)$ when all the primes   $p|d$ are congruent to $1\pmod 8$. For this, we need the following lemmas.
  
  \begin{lm}\label{lm e_p with p=1 mod 8}
  	Let $p$ be a prime such that  $p\equiv 1 \pmod 8$.  Then $e_p=0 \text{ or } 1$. More precisely,
  	\begin{enumerate}[\rm1.]
  		\item $e_p= 0$  if and only if $p\equiv 1 \pmod{16}$ and $\left(\frac{2}{ p}\right)_4=\left(\frac{p}{ 2}\right)_4$.
  		\item $e_p= 1$ if and only if
  		$p\equiv 9 \pmod{16}$ or $\left(\frac{2}{ p}\right)_4\not=\left(\frac{p}{ 2}\right)_4.$
  	\end{enumerate}
  \end{lm}
  \begin{proof}
  	%Recall that,  by Theorem \ref{prop number of ramified primes}, $r_2(p)=4-1-e=3-e$.
  	Let $p\in \mathfrak{p}_{\mathbb{Q}(i)}\subset \mathfrak{p}_K$ be two prime ideals of $\mathbb{Q}(i)$ and $K$ respectively. We shall  calculate the two symbols $\left(\frac{\zeta_8,\, p}{\mathfrak{p}_K}\right)$ and $\left(\frac{\varepsilon_2,\, p}{\mathfrak{p}_K}\right)$.
  	We have
  	$$\begin{array}{lll}	
  	\left(\frac{\zeta_8,\, p}{\mathfrak{p}_K}\right)&=&\left(\frac{1+i,\, p}{\mathfrak{p}_K}\right)\left(\frac{\frac{1}{\sqrt{2}},\, p}{\mathfrak{p}_K}\right)
  	= \left(\frac{1+i}{\mathfrak{p}_K}\right)\left(\frac{\sqrt{2},\, p}{\mathfrak{p}_K}\right)\\
  	&=&\left(\frac{1+i}{\mathfrak{p}_K}\right)\left(\frac{\sqrt{2},\, p}{\mathfrak{p}_K}\right)
  	=\left(\frac{2}{a+b}\right)\left(\frac{\sqrt{2},\, p}{\mathfrak{p}_K}\right),
  	\end{array}$$
  	where $a$ and $b$ are two integers such that $p=a^2+b^2$ (see \cite[page 154]{lemmermeyer2013reciprocity}). Since $\left(\frac{\varepsilon_2\sqrt{2},\, p}{\mathfrak{p}_K}\right)=\left(\frac{\varepsilon_2\sqrt{2}}{\mathfrak{p}_K}\right)=\left(\frac{\varepsilon_2\sqrt{2}}{\mathfrak{p}_{\mathbb{Q}(\sqrt{2})}}\right)=(-1)^{\frac{p-1}{8}},$ where $\mathfrak{p}_{\mathbb{Q}(\sqrt{2})}$ is an ideal of $\mathbb{Q}(\sqrt{2})$ lying over $p$ (see \cite[page 21]{brown19782}),  then
  	$$\begin{array}{ll}
  	\left(\frac{\sqrt{2},\, p}{\mathfrak{p}_K}\right)=(-1)^{\frac{p-1}{8}}\left(\frac{ \varepsilon_2,\, p}{\mathfrak{p}_K}\right)
  	=(-1)^{\frac{p-1}{8}}\left(\frac{ \varepsilon_2}{\mathfrak{p}_{  K}}\right)=(-1)^{\frac{p-1}{8}}\left(\frac{ \varepsilon_2}{\mathfrak{p}_{\mathbb{Q}(\sqrt{2})}}\right).
  	\end{array}$$
  	On the other hand, by  \cite[page 323]{kaplan76} and  \cite[page 160]{lemmermeyer2013reciprocity}, we have
  	$$\begin{array}{ll}\left(\frac{2}{a+b}\right)=\left(\frac{2}{p}\right)_4\left(\frac{p}{2}\right)_4=\left(\frac{ \varepsilon_2}{\mathfrak{p}_{\mathbb{Q}(\sqrt{2})}}\right).\end{array}$$
  	Hence
  	$$\begin{array}{ll}\left(\frac{\zeta_8,\, p}{\mathfrak{p}_K}\right)=	(-1)^{\frac{p-1}{8}}.\end{array}$$
  	If $\zeta_{8}$ and $\varepsilon_2$ are not norms in $L_p/K$,  we claim that $\overline{\zeta_{8}}=\overline{\varepsilon_2}$ in $E_{K}/(E_{K}\cap N(L_p))$. Indeed  $\left(\frac{\zeta_8,\, p}{\mathfrak{p}_{K}}\right)$ and $\left(\frac{\varepsilon_2,\, p}{\mathfrak{p}_K}\right)$ do not depend on    $\mathfrak{p}_K$   so       $\left(\frac{\zeta_{8}  \varepsilon_2,\, p}{\mathfrak{p}_{K}}\right)=1$ for all $\mathfrak{p}_K$. It follows that $\overline{\zeta_{8}  \varepsilon_2} =\overline{1}$, then  $\overline{\zeta_{8} }=\overline{  \varepsilon_2}$ as claimed.
  \end{proof}
  
  \noindent From the previous proof, we have the following remarks.
  \begin{rema}\label{corr of thm de Ssi Zekhnini}
  	Let $p$ be a prime such that $p\equiv 1 \pmod 8$. Then, for any prime ideal $\mathfrak p_K$ of $K$ above $p$, we have
  	$\displaystyle\left(\frac{\zeta_8,\, p}{\mathfrak{p}_K}\right)=	(-1)^{\frac{p-1}{8}} \text{and } \left(\frac{ \varepsilon_2,\, p}{\mathfrak{p}_K}\right)=\left(\frac{2}{p}\right)_4\left(\frac{p}{2}\right)_4.$
  %	Thus
  %	\begin{enumerate}[\rm1.]
  	%	\item $\varepsilon_2$ is a norm in $L_p/K$ if and
  %		only if $\left(\frac{2}{p}\right)_4=
  %		\left(\frac{p}{ 2}\right)_4.$
  %		\item $\zeta_{8}$ is a norm in $L_p/K$ if and only if $p\equiv 1 \pmod{16}$.
  %	\end{enumerate}
  %	Furthermore, $\zeta_{8}$ $($resp. $\varepsilon_2$$)$ is a norm in $L_p/K$ if and only if $\left(\frac{\zeta_8,\, p}{\mathfrak{p}_{K}}\right)=1$ $($resp. $\left(\frac{\varepsilon_2,\, p}{\mathfrak{p}_K}\right)=1$$)$ for some prime ideal $\mathfrak{p}_K$ of $K$ above $p$. In fact, 
 % 	the value of $\left(\frac{\zeta_8,\, p}{\mathfrak{p}_{K}}\right)$ $($resp. $\left(\frac{\varepsilon_2,\, p}{\mathfrak{p}_K}\right))$ is  the same for any prime ideal  $\mathfrak{p}_K$ of $K$ above $p$.
  \end{rema}
  
  \begin{lm}\label{lm  les unites sont dans le mm classe}
  	Let $d=p_1...p_r>2$ be an odd  composite square-free positive integer  such that $p_i\equiv 1 \pmod 8$ for all $i$. Consider the  following assertions:
  	\begin{enumerate}[\rm$(a)$]
  		\item  $\zeta_8$ and $\varepsilon_2$ are not norms in $L_{p_i}/K$ and $L_{p_j}/K$  respectively  for some $i\not=j$.
  		\item  $\zeta_{8}$ is a norm in $L_{p_j}/K$ or  $\varepsilon_2$ is a norm in $L_{p_i}/K$.
  		\item $\zeta_8$ and $\varepsilon_2$ are not norms in $L_d/K$.
  		\item $\overline{\zeta_{8}}\not=\overline{\varepsilon_2}$ in $E_{K}/(E_{K}\cap N(L_d))$.
  	\end{enumerate}
  	Then  $($a$)$ and $($b$)$ hold if and only if  $($c$)$ and $($d$)$ hold.	
  \end{lm}
  \begin{proof}
  	Assume that  $(a)$ and $(b)$ hold, then
  	$\left(\frac{\zeta_8,\, p_i}{\mathfrak{p}_{ i}}\right)=-1=\left(\frac{\varepsilon_2,p_j}{\mathfrak{p}_{ j}}\right),$
  	where $\mathfrak p_i$ (resp. $\mathfrak p_j$ ) is a prime ideal of $K$ lying above $p_i$ (resp. $p_j$).
  	So by Lemma \ref{lemma norm residue symbol pp} we get
  	$$\left(\frac{\zeta_8,\, d}{\mathfrak{p}_{ i}}\right)=\left(\frac{\zeta_8,\, p_i}{\mathfrak{p}_{ i}}\right)=-1=\left(\frac{\varepsilon_2,\, d}{\mathfrak{p}_{ j}}\right)=\left(\frac{\varepsilon_2,p_j}{\mathfrak{p}_{ j}}\right),$$
  	and  $(c)$ follows.
  	If  $\left(\frac{\varepsilon_2,\, p_i}{\mathfrak{p}_{ i}}\right)=1$, then
  	$\left(\frac{\zeta_{8}\varepsilon_2,\, d}{\mathfrak{p}_{ i}}\right)=\left(\frac{\zeta_{8}\varepsilon_2,\, p_i}{\mathfrak{p}_{ i}}\right)=\left(\frac{\zeta_{8},\, p_i}{\mathfrak{p}_{ i}}\right)\left(\frac{\varepsilon_2,\, p_i}{\mathfrak{p}_{ i}}\right)=-1$. Thus $\overline{\zeta_{8}\varepsilon_2}\not=\overline{1}$. Hence $\overline{\zeta_{8}}\not=\overline{\varepsilon_2}$. (We similarly treat the case $\left(\frac{\zeta_{8},\, p_i}{\mathfrak{p}_{ i}}\right)=1$). So the assertion  $(d)$.
  	
  	Conversely, suppose    $(c)$ and $(d)$ hold. Since $\zeta_8$ and $\varepsilon_2$ are not norms in $L_d/K$, then there exist $i$ and $j$ such that $\left(\frac{\zeta_8, p_i}{\mathfrak{p}_{ i}}\right)=-1$ and
  	$\left(\frac{\varepsilon_2,p_j}{\mathfrak{p}_{ j}}\right)=-1$. Suppose this is true only for $i=j$, then for all $k$, we get $\left(\frac{\zeta_8,
  		p_k}{\mathfrak{p}_{ k}}\right)\left(\frac{\varepsilon_2,p_k}{\mathfrak{p}_{ k}}\right)=\left(\frac{\zeta_8\varepsilon_2,
  		p_k}{\mathfrak{p}_{ k}}\right)=1$. So $\zeta_8\varepsilon_2$ is a norm in $L_d/K$, which contradicts $(d)$. Thus $(a)$ holds. Suppose that $(b)$ is not verified, then  for all $i\not=j$ satisfying  $(a)$ and   not verifying $(b)$, we have 
  	$$ \left(\frac{\zeta_8,
  		p_i}{\mathfrak{p}_{ i}}\right)=\left(\frac{\varepsilon_2,\,
  		p_i}{\mathfrak{p}_{ i}}\right)=-1=\left(\frac{\zeta_8,
  		p_j}{\mathfrak{p}_{ j}}\right)=\left(\frac{\varepsilon_2,
  		p_j}{\mathfrak{p}_{ j}}\right).$$
  	\noindent Thus  for all $k$, $\left(\frac{\zeta_8\varepsilon_2,p_k}{\mathfrak{p}_{ k}}\right)=1$, i.e., $\left(\frac{\zeta_8\varepsilon_2,\, d}{\mathfrak{p}}\right)=1$ for all prime $\mathfrak{p}$ of $K$ ramified in $L_{d}$. It follows that,   $\zeta_8\varepsilon_2$ is a norm in $L_{d}/K$, which contradicts  $(d)$. Hence $(b)$ holds too, which ends the proof of the lemma.
  \end{proof}
  
  \noindent  Now we give an  analogous of  Lemma \ref{lm e_p with p=1 mod 8} for a composite square-free integer $d>2$.  	
  \begin{lm}\label{lm e_d with pi=1 mod 8}
  	Let $d=p_1...p_r>2$ be a composite odd  square-free integer such that $p_i\equiv 1 \pmod 8$  for all $ i\in I=\{1, \cdots, r\}$. Then
  	\begin{enumerate}[\rm1.]
  		\item $e_d=0 \Longleftrightarrow\forall i\in I, p_i \equiv 1 \pmod{16}\text{ and }\left(\frac{2}{ p_i}\right)_4=\left(\frac{p_i}{ 2}\right)_4.$
  		\item $e_d=1$ if and only if one of the following assertions holds
  		\begin{enumerate}[\rm i.]
  			\item $\forall i\in I, p_i \equiv 1 \pmod{16}\text{ and } \exists j\in I, \left(\frac{2}{ p_j}\right)_4\not=\left(\frac{p_j}{ 2}\right)_4,$
  			\item $\exists i\in I, p_i \equiv 9 \pmod{16} \text{ and } \forall j\in I, \left(\frac{2}{ p_j}\right)_4=\left(\frac{p_j}{ 2}\right)_4,$
  			\item $\exists(i,j)\in I^2, \left[p_i \equiv 9 \pmod{16} \text{ and }  \left(\frac{2}{ p_j}\right)_4\not=\left(\frac{p_j}{ 2}\right)_4\right]$ and
  			all the couples $(i,j)$ satisfying the last  condition   satisfy also   $\left[\left(\frac{2}{ p_i}\right)_4\not=\left(\frac{p_i}{ 2}\right)_4  \text{ and }    p_j \equiv 9 \pmod{16}\right]$.
  		\end{enumerate}
  		\item $e_d=2$ if and only if there exist    $i\not=j\in I$ such that
  		$p_i \equiv 9 \pmod{16}$, $\left(\frac{2}{ p_j}\right)_4\not=\left(\frac{p_j}{ 2}\right)_4$ and $\left[
  		\left(\frac{2}{ p_i}\right)_4=\left(\frac{p_i}{ 2}\right)_4\text{ or }p_j \equiv 1 \pmod{16}\right]$.
  	\end{enumerate}
  \end{lm}
  \begin{proof} We have:
  	\begin{enumerate}[\rm $\bullet$]
  		\item  The first assertion is deduced from  Lemma \ref{lemma norm residue symbol pp}, Remark \ref{corr of thm de Ssi Zekhnini}  and the fact that $e_d=0$ if and only if both $\zeta_{8}$ and $\varepsilon_2$ are  norms in $L_{d}/K$.
  		\item The second assertion  is deduced too from Lemmas   \ref{lemma norm residue symbol pp}, \ref{lm  les unites sont dans le mm classe}, Remark \ref{corr of thm de Ssi Zekhnini} and  the fact that $e_d=1$ if and only if \{($\zeta_8$ is a norm and  $\varepsilon_2$ is not) or  ($\zeta_8$ is not a norm and  $\varepsilon_2$ is) or (both $\zeta_8$ and  $\varepsilon_2$ are not norms and $\overline{  \varepsilon_2} =\overline{\zeta_{8}}$)\}.
  		\item The last assertion is a result of Lemmas \ref{lemma norm residue symbol pp}, \ref{lm  les unites sont dans le mm classe} and the  fact that $e_d=2$ if and only if $\zeta_8$ and  $\varepsilon_2$ are not norms and $\overline{  \varepsilon_2} \not=\overline{\zeta_{8}}$.
  	\end{enumerate}	
  \end{proof}
  
  Now, we can easily deduce the following theorem.
  \begin{theorem}\label{thm r_2(d) with pi=1 mod 8}
  	Let $d>2$ be an odd   square-free integer such that all its prime divisors are congruent to $1\pmod8$. If  $r$ is the number of these distinct primes divisors, then
  	$$r_2(d)=4r-1-e_d,$$
  	where $e_d$ is given by Lemmas $\ref{lm e_p with p=1 mod 8}$ and $\ref{lm e_d with pi=1 mod 8}$.	
  \end{theorem}

  We close this subsection with some numerical examples.
  \begin{exams}~\
  	\begin{enumerate}[\rm 1.]
  		\item For $d=73\cdot89\cdot97$, we have $73\equiv89\equiv 9\pmod{16}$, $97\equiv 1\pmod{16}$, $\left(\frac{73}{ 2}\right)_4=\left(\frac{89}{ 2}\right)_4=-\left(\frac{97}{ 2}\right)_4=-1$ and
  		$\left(\frac{2}{ 73}\right)_4=\left(\frac{2}{ 89}\right)_4=-\left(\frac{2}{ 97}\right)_4=1$. Thus,  by   Theorem \ref{thm r_2(d) with pi=1 mod 8}, the rank of the $2$-class group of $L_{d}:=\mathbb{Q}(\sqrt{2},i, \sqrt{73\cdot89\cdot97})$ equals\\  $4\cdot3-3=9$ (see the third item of Lemma \ref{lm e_d with pi=1 mod 8}), and the class group of $L_d$ by PARI/GP is of type $(1224, 8, 4, 4, 4, 2, 2, 2, 2)$.
  		\item 	For $d=73\cdot89\cdot113$, we have $73\equiv89\equiv 9\pmod{16}$, $113\equiv 1\pmod{16}$, $\left(\frac{73}{ 2}\right)_4=\left(\frac{89}{ 2}\right)_4=-\left(\frac{113}{ 2}\right)_4=-1$ and   $\left(\frac{2}{ 73}\right)_4=\left(\frac{2}{ 89}\right)_4=\left(\frac{2}{ 113}\right)_4=1$. So by   Theorem \ref{thm r_2(d) with pi=1 mod 8}, the rank of the $2$-class group of $L_{d}:=\mathbb{Q}(\sqrt{2},i, \sqrt{73\cdot89\cdot113})$ (resp. $\mathbb{Q}(\sqrt{2},i, \sqrt{73\cdot113})$ ) is $4\cdot3-2=10$ (resp. $4\cdot2-2=6$) (see the second item of Lemma \ref{lm e_d with pi=1 mod 8}), and the class group of $L_d$ by PARI/GP is of type $(384, 32, 2, 2, 2, 2, 2, 2, 2, 2)$ (resp. $(912, 2, 2, 2, 2, 2)$).
  		\item For $d=353\cdot257\cdot113$, we have $353\equiv 257  \equiv 113 \equiv 1\pmod{16}$ and $\left(\frac{353}{ 2}\right)_4=\left(\frac{257}{ 2}\right)_4=\left(\frac{113}{ 2}\right)_4=\left(\frac{2}{ 353}\right)_4=\left(\frac{2}{ 257}\right)_4=\left(\frac{2}{ 113}\right)_4=1$. So by   Theorem \ref{thm r_2(d) with pi=1 mod 8}, the rank of the $2$-class group of $L_{d}:=\mathbb{Q}(\sqrt{2},i, \sqrt{353\cdot257\cdot113})$ (resp. $\mathbb{Q}(\sqrt{2},i, \sqrt{257\cdot113})$) is $4\cdot3-1=11$ (resp. $4\cdot2-1=7$)  (see the first item of Lemma \ref{lm e_d with pi=1 mod 8}), and the class group of $L_d$ by PARI/GP is of type $(408, 204, 2, 2, 2, 2, 2, 2, 2, 2, 2)$ (resp. $(4368, 8, 2, 2, 2, 2, 2)$).
  	\end{enumerate}
  \end{exams}

  \subsection{Case 2: The  prime divisors of $d$ are not in the same coset of $\mathbb{Z}/8\mathbb{Z}$}~\\	
  In this subsection, we will make use of  Lemmas \ref{lemma norm residue symbol pp}, \ref{lm symbols 1}  and Remark \ref{corr of thm de Ssi Zekhnini}  to determine $r_2(d)$ for any odd composite square-free integer $d>2$ for which the prime divisors are not in the same coset of  $\mathbb{Z}/8\mathbb{Z}$. Since the number $t_d$ of prime ideals of $K$  ramified in $L_d$ is    determined by  Theorem \ref{prop number of ramified primes}, we shall give the rank of the $2$-class group of $L_d$ in terms of  $t_d$.

  \begin{theorem}\label{thm rank in general case}
  	Let $d>2$ be an odd composite   square-free integer such that the primes dividing $d$ are not in the same coset $\pmod 8$. 
  	\begin{enumerate}[\rm1.]
  		\item   If there exist two prime divisors $p_1 $ and   $p_2$ of $d$  such that $p_1\equiv -p_2\equiv 5\pmod 8$, then
  		$r_2(d)=t_d-3.$
  		
  		\item If $d$  is divisible by a prime congruent to $3\pmod 8$ and none of the other prime divisors of $d$ is congruent to $5\pmod 8$, then $r_2(d)=t_d-2$ or $t_d-3$. More precisely,
  		$r_2(d)= t_d-3$ if and only if there exists a prime $p\equiv 1\pmod 8$ dividing $d$ such that  $  \left(\frac{2}{ p}\right)_4=-1$.
  		
  		\item If $d$  is divisible by a prime congruent to $5\pmod 8$ and none of the other prime divisors of $d$ is congruent to $3\pmod 8$, then $r_2(d)=t_d-2$ or $t_d-3$. More precisely,
  		$r_2(d)= t_d-3$ if and only if there exists a prime $p\equiv 1\pmod 8$ dividing $d$ such that  $ \left(\frac{2}{ p}\right)_4\not=\left(\frac{p}{ 2}\right)_4$.
  		\item If all the primes dividing $d$ are congruent to $\pm1\pmod 8$, then
  		$r_2(d)=t_d-1-e_{ d_1}$, where $d_1$ is the product of all the primes $p|d$ such that $p\equiv 1\pmod 8$.
  		Note that 	$e_{{d_1}}$ is given by   Lemmas $\ref{lm e_p with p=1 mod 8}$ and $\ref{lm e_d with pi=1 mod 8}$.
  	\end{enumerate}
  \end{theorem}

  \begin{proof}
  	By   formula \eqref{egalite du 2-rang   =t-1-e}, the rank of the $2$-class group of $L_d$ is
  	$r_2(d)=t_{d}-1-{e_{d}}.$ As $K$ is a biquadratic number field,  then $e_d\in \{0,1,2\}$. On the other hand,  the symbols $\left(\frac{\zeta_8,\,
  		d}{\mathfrak{p}}\right) $ and $\left(\frac{\varepsilon_2,\,
  		d}{\mathfrak{p}}\right)$ are trivial for any prime ideal $\mathfrak{p}$ of $K$ lying over a prime $p\equiv 7\pmod 8$ (see  Lemmas \ref{lemma norm residue symbol pp}, \ref{lm symbols 1}) so we will  ignore them.
  	\begin{enumerate}[\rm1.]
  		\item Let $\mathfrak p_1$ be a prime of $K$ above $p_1$. By Lemmas \ref{lemma norm residue symbol pp} and  \ref{lm symbols 1}, the units $\zeta_{8}$ and $\varepsilon_2$ are not norms in $L_d/K$ and
  		$\left(\frac{\zeta_8\varepsilon_2,\,
  			d}{\mathfrak{p}_{1}}\right)=\left(\frac{\varepsilon_2,\,
  			{p}_{1}}{\mathfrak{p}_{1}}\right)\left(\frac{\zeta_8,\,
  			{p}_{1}}{\mathfrak{p}_{1}}\right)=-1$, so
  		$\overline{\zeta_{8}}\not=\overline{\varepsilon_2}$ in $E_{K}/(E_{K}\cap N(L_d))$. Thus $e_d=2$. Hence the first item.
  		
  		\item Let  $p_1$ be a prime  dividing  $d$ such that  $p_1\equiv 3\pmod 8$. Assume that $d$ is divisible by  prime   $p_2\equiv 1\pmod 8$. Denote by $\mathfrak p_1$ and $\mathfrak p_2$ two prime ideals of $K$ lying over $p_1$ and $p_2$ respectively.
  		By Lemmas \ref{lemma norm residue symbol pp}, \ref{lm symbols 1} and  Remark \ref{corr of thm de Ssi Zekhnini},  we have
  		$\left(\frac{\zeta_8,\, d}{\mathfrak{p}_{1}}\right)=\left(\frac{\zeta_8,\, p_1}{\mathfrak{p}_{1}}\right)=-1$, $\left(\frac{\varepsilon_2,\, d}{\mathfrak{p}_{1}}\right)=\left(\frac{\varepsilon_2,\, p_1}{\mathfrak{p}_{1}}\right)=-1,$
  		$\left(\frac{\zeta_8,\, d}{\mathfrak{p}_{2}}\right)=\left(\frac{\zeta_8,\, p_2}{\mathfrak{p}_{2}}\right)=(-1)^{\frac{p_{2}-1}{8}}$ and similarly $\left(\frac{\varepsilon_2,\, d}{\mathfrak{p}_{2}}\right)=\left(\frac{2}{ p_{2}}\right)_4\left(\frac{p_{2}}{ 2}\right)_4.$
  		Hence $\zeta_8$ and $\varepsilon_2$ are not   norms in $L_d/K$, and so $e_{d}\not=0$, which implies that $e_{d}\in \{1,2\}$. We have   $e_{d}=2$ if and only if  $\overline{\zeta_{8}}\not=\overline{\varepsilon_2}$ in $E_{K}/(E_{K}\cap N(L_d))$ and this, by Lemma \ref{lm e_d with pi=1 mod 8}, can only happen for primes $\equiv 1\pmod 8$. Since 	$\left(\frac{\zeta_8\varepsilon_2,\, d}{\mathfrak{p}_{2}}\right)=(-1)^{\frac{p_{2}-1}{8}}\left(\frac{2}{ p_{2}}\right)_4\left(\frac{p_{2}}{ 2}\right)_4=\left(\frac{2}{ p_{2}}\right)_4$, then $e_{d}=2$ if and only if $\left(\frac{2}{ p}\right)_4=-1$ for some prime $p|d$ such that  $p\equiv 1\pmod 8$.
  	 The second item follows.
  		
  		We similarly prove the third item. The fourth item is immediate.	
  \end{enumerate}\end{proof}

  We close this subsection with the following numerical examples.
  \begin{exam}~\
  	\begin{enumerate}[\rm1.]
  		\item For  $d=7\cdot3\cdot113$ (resp. $d=7\cdot3\cdot17$), we have   $\left(\frac{113}{ 2}\right)_4=\left(\frac{2}{113}\right)_4=-\left(\frac{2}{17}\right)_4=1$. So by the second item of the previous theorem the rank of the $2$-class group of $L_d=\mathbb{Q}(i, \sqrt{2}, \sqrt{7\cdot3\cdot113} )$ (resp. $L_d=\mathbb{Q}(i, \sqrt{2}, \sqrt{7\cdot3\cdot17} )$) is $8-2=6$ (resp.  $8-3=5$),  and the class group of $L_d$ by PARI/GP is of type $(42, 2, 2, 2, 2, 2)$ (resp. $(12, 2, 2, 2, 2)$).
  		\item For $d=7\cdot5\cdot17$ (resp. $d=7\cdot5\cdot113$), we have   $\left(\frac{17}{2}\right)_4=-\left(\frac{2}{17}\right)_4=\left(\frac{113}{2}\right)_4=\left(\frac{2}{113}\right)_4=1$. So, by  the third item of the previous theorem the rank of the $2$-class group of $L_d=\mathbb{Q}(i, \sqrt{2}, \sqrt{7\cdot5\cdot17} )$ is $8-3=5$ (resp. $L_d=\mathbb{Q}(i, \sqrt{2}, \sqrt{7\cdot5\cdot113} )$ is $8-2=6$),  and the class group of $L_d$ by PARI/GP is of type $(20, 2, 2, 2, 2)$ (resp. $(42, 6, 2, 2, 2, 2)$).
  		\item For $d=7\cdot17$ (resp. $d=7\cdot113$), we have   $\left(\frac{17}{2}\right)_4=-\left(\frac{2}{17}\right)_4=\left(\frac{113}{2}\right)_4=\left(\frac{2}{113}\right)_4=1$.  Then,  by the last item of the previous theorem the rank of the $2$-class group of $L_d=\mathbb{Q}(i, \sqrt{2}, \sqrt{7\cdot17} )$ is $6-1-1=4$ (resp. $L_d=\mathbb{Q}(i, \sqrt{2}, \sqrt{7\cdot113} )$ is $6-1-0=5$),  and the class group of $L_d$ by PARI/GP is of type $(20, 2, 2, 2)$ (resp. $(64, 2, 2, 2, 2)$).
  	\end{enumerate}
  \end{exam}

  \section{Applications}\label{section sur les applications}
  In this section, we  will determine the integers  $d$ such that   the  $2$-class group of $L_d$  is trivial, cyclic or isomorphic to  $\mathbb{Z}/2\mathbb{Z}\times\mathbb{Z}/2\mathbb{Z}$. For this, we have to recall  the following results.
  \begin{lm}[\cite{lemmermeyer1995ideal}]\label{lm Kurodas formula}
  	Let $k'/k$ be a  biquadratic extension of $\mathrm{CM}$-type, then
  	$$h(k')=\frac{Q_{k'}}{Q_{k_1}Q_{k_2}}\cdot\frac{\omega_{k'}}{\omega_{k_1}\omega_{k_2} }\cdot\frac{h(k_1)h(k_2)h(k'^+)}{h(k)^2}\cdot$$
  	Where $k_1, k_2$ and $k'^+$ are the three sub-extensions of $k'/k$.
  \end{lm}
  
  \noindent The result below, gives the class number of a multiquadratic number field in terms of those of its quadratic subfields. 	
  \begin{prop}[\cite{wada}]\label{wada's f.}
  	Let $k$ be a multiquadratic number field of degree $2^n$, $n\in\mathds{N}$,  and $k_i$ the $s=2^n-1$ quadratic subfields of $k$. Then
  	$$h(k)=\frac{1}{2^v} (E_k: \prod_{i=1}^{s}E_{k_i}) \prod_{i=1}^{s}h(k_i),$$
  	with $$v=\left\{ \begin{array}{cl}
  	n(2^{n-1}-1); &\text{ if } k \text{ is real, }\\
  	(n-1)(2^{n-2}-1)+2^{n-1}-1 & \text{ if } k \text{ is imaginary.}
  	\end{array}\right.$$
  \end{prop}

  \begin{lm}[\cite{mccall1995imaginary}]\label{lm de parry}
  	Let $S$ be  the set of odd primes that are ramified in $\mathbb{Q}(\sqrt d)$ and  $S_0$ its subset 
  	consisting of those primes that are congruent to $1\pmod 4$. Let $s$ and $s_0$ be the cardinality  of $S$ and $S_0$ respectively. Then
  	the rank of the $2$-class group of $k=\mathbb Q(\sqrt d ,i)$ equals		
  	\begin{itemize}
  		\item $s+s_0$ if $d$ is even and $p\equiv 1\pmod 8$ for all $ p\in S_0$ $($the case  $S_0=\emptyset$ is included here$)$.
  		\item $s+s_0-1$ if $d$ is even and there exists $p\in S_0$ satisfying $p\equiv 5 \pmod 8$,\\
  		or $d$ is odd and $p\equiv 1\pmod 8$, for all $  p\in S_0$ $($the case  $S_0=\emptyset$ is included here$)$.
  		\item $s+s_0-2$ if $d$ is odd and there exists $p\in S_0$ satisfying $p\equiv 5\pmod 8$.
  	\end{itemize}
  %The cases $S_0=\empty$is considered in the case for all $p\equiv 1\pmod 8$, and we take $s_0=0$.
  \end{lm}

  \subsection{Fields $L_d$ with trivial or cyclic $2$-class group}\label{subsection    2-class group  is trivial or cyclic}
  Using   Theorem \ref{prop number of ramified primes} and the theorems in the  previous sections, one can  easily deduce the following  results.
  \begin{theorem}
  	The $2$-class group   of $L_{d}$ is trivial if and only if $d$ is a prime  congruent to either $3$ or $5\pmod8$.
  \end{theorem}
  \begin{theorem}
  	The $2$-class group    of $L_{d}$ is cyclic nontrivial if and only if $d$  takes one of the following forms
  	\begin{enumerate}[\rm1.]
  		\item $d= q\equiv 7\pmod 8$ is a prime,
  		\item $d=qp$, where   $q\equiv 3\pmod 8$ and $p\equiv 5\pmod 8$ are primes.
  	\end{enumerate}
  \end{theorem}	
  
  \subsection{Fields $L_d$ with $2$-class group of rank $2$}\label{sous-section les d tels que le rang = 2}
  By Theorem \ref{prop number of ramified primes} and the  theorems in the  previous sections, we  easily deduce the following  results.
  
  \begin{theorem}\label{thm 2-rank = 2}
  	The rank of the $2$-class group    of $L_{d}$ equals $2$  if and only if $d$  takes one of the following forms
  	\begin{enumerate}[\rm1.]
  		\item $d=q_1q_2$, with   $q_1\equiv q_2 \equiv3\pmod8$,
  		\item $d=p_1p_2$, with   $ p_1\equiv p_2\equiv5\pmod8$,
  		\item $d=q_1q_2$, with   $q_1\equiv3\pmod8$ and $q_2\equiv7\pmod8$,
  		\item $d=pq$, with   $p\equiv5\pmod8$ and $q\equiv7\pmod8$,
  		\item $d=p\equiv1\pmod8 $ is a prime satisfying $\left[p\equiv 9 \pmod{16} \text{ or }\left(\frac{2}{ p}\right)_4\not=\left(\frac{p}{ 2}\right)_4\right]$,\\
  		
  	\end{enumerate}
  	\vspace{-0.5cm}	where $ p_i, q_i,p$ and $q$ are prime integers.
  \end{theorem}
  \subsection{Fields $L_d$ with $2$-class group of type $(2, 2)$}	
  Now we shall determine the integers $d$ for which  the $2$-class group of  $L_d$ is of type $(2, 2)$. The main theorem of this subsection is the following. For the proof see Propositions \ref{thm d=p equiv 1 mod 8, when clL=(2,2)}, \ref{4}, \ref{5}, \ref{prop d=pq  cl=(2,2)} and \ref{prop d=q1q2  cl=(2,2)} below.
  \begin{theorem}
  	Let $d>2$ be an odd    square-free integer. The $2$-class group of $L_d=\QQ(\sqrt d, \sqrt2, i)$ is of type $(2,2)$ if and only if $d$ takes one of the following forms
  	\begin{enumerate}[\rm1.]
  		\item $d=p$, with $p\equiv 1 \pmod{16}$  and $\left(\frac{2}{p}\right)_4\not=\left(\frac{p}{2}\right)_4$.
  		\item $d=pq$, with $p\equiv 5\pmod 8$, $q\equiv 7\pmod 8$ and $\left(\frac{p}{q}\right)=-1$.
  		\item $d=q_1q_2$, with $q_1\equiv 3\pmod 8$, $q_2\equiv 7\pmod 8$ and $\left(\frac{q_1}{q_2}\right)=-1$.
  		
  	\end{enumerate}
  	where $q_i,q$ and $p$ are primes.
  \end{theorem}

 % To prove the main theorem of this subsection, we  need some preliminary results.

  \noindent To prove this theorem, we shall check all the items of Theorem \ref{thm 2-rank = 2}.
  \begin{prop}\label{thm d=p equiv 1 mod 8, when clL=(2,2)}
  	Let $p$ be a prime such that $p\equiv 1\pmod 8$. Then
  	$$
  	\mathrm{Cl}_2(L_p)=(2,2) \text{ if and only if } p\equiv 1 \pmod {16} \text{ and } \left(\frac{2}{p}\right)_4\not=\left(\frac{p}{2}\right)_4.
  	$$
  	Moreover,  $h_2(L_p)=h_2(-2p)$ if and only if $\left(\frac{2}{p}\right)_4\not=\left(\frac{p}{2}\right)_4$.
  \end{prop}
 \begin{proof}
 	Let $p\equiv 1 \pmod 8$ be a prime.
 	Set  $L_p^+=\mathbb{Q}(\sqrt{2}, \sqrt{p})$, $K=\mathbb{Q}(\sqrt{2}, i)$  and $K'=\mathbb{Q}(\sqrt{2}, \sqrt{-p})$. By applying Lemma \ref{lm Kurodas formula} to the extension $L_p/\mathbb{Q}(\sqrt{2})$, we have
 	\begin{eqnarray*}
 		h(L_p)=\frac{Q_{L_p}}{Q_{K}Q_{K'}}\frac{\omega_{L_p}}{\omega_{K} \omega_{K'}}
 		\frac{h( L_{p}^+)h(K)h(K')}{h(\mathbb{Q}(\sqrt{2}))^2}\cdot
 	\end{eqnarray*}
 	
 	We have $ h(\mathbb{Q}(\sqrt{2}))=h(K)=1$. By \cite[Théorème 3]{taous2008}, $Q_{L_p}=1$ and by Lemma \ref{lm unite zeta8} $Q_{K}=1$. Since $ \omega_{L_p}= \omega_{K}=8$ and $ \omega_{K'}=2$, then by passing to the $2$-part in the above equality we get
 	\begin{eqnarray}\label{eq 2}
 	h_2(L_p)=\frac{1}{2Q_{K'}}h_2( L_{p}^+)h_2(K')\cdot
 	\end{eqnarray}

 	As $\varepsilon_2$ has  a negative norm, so by the  item $(2)$ of Section $3$ of \cite{AZ99} we obtain that
 	$E_{K'}=\langle -1, \varepsilon_2  \rangle$. This in turn implies that $q(K')=Q_{K'}=1$. From which we infer, by Proposition \ref{wada's f.}, that
 	$h_2(K')=\frac{1}{2} \cdot 1 \cdot h_2(2) h_2(-p)  h_2(-2p)=\frac{1}{2}     h_2(-p)  h_2(-2p)$. It follows, by the equality \eqref{eq 2}, that
 	
 	\begin{eqnarray}\label{eq 3}
 	h_2(L_p)=\frac{1}{4}h_2( L_{p}^+)  h_2(-p)  h_2(-2p)\cdot
 	\end{eqnarray}
 	
 	Keep the notations of  \cite[Theorem 2]{ezrabrown}, by this theorem we have 
 	$h_2(-p)=4$ if and only if   $\genfrac(){}{0}{e}{p}=-1.$ 
 	From the  proof of    \cite[Theorem 1]{ezrabrown}, one deduces easily that 
 	$\genfrac(){}{0}{e}{p}=\genfrac(){}{0}{2}{p}_4  \genfrac(){}{0}{p}{2}_4 $. Therefore,   $h_2(-p)=4$ if and only if $\genfrac(){}{0}{2}{p}_4 \not=\genfrac(){}{0}{p}{2}_4$.
 	
 	Thus by Theorem \ref{thm 2-rank = 2}, we   have only  two  cases to check.
 	\begin{enumerate}[\rm $\bullet$]
 		\item If $\left(\frac{2}{p}\right)_4\not=\left(\frac{p}{2}\right)_4=(-1)^{\frac{p-1}{8}}$, then by
 		\cite[Theorem 2]{kuvcera1995parity},
 		$h( L_{p}^+)$ is odd. So
 		$$h_2(L_p)=\frac{1}{4}h_2(-p)h_2(-2p).$$
 		As  $h_2(-p)=4$, then
 		$h_2(L_p)=h_2(-2p).\label{formula de 2nc}$
 		From   \cite{Scholz1935} we deduce that $h_2(-2p)=4$ if and only if $\left(\frac{p}{2}\right)_4=1$.
 		\item If $p\equiv 9\pmod{16}$ and $\left(\frac{2}{p}\right)_4=\left(\frac{p}{2}\right)_4=-1$, then by  \cite[Théorème 10]{taous2008},   the rank of   the $4$-class group is equal to $1$  and from Theorem \ref{thm r_2(d) with pi=1 mod 8}, we infer that $h_2(L_p)$ is divisible by $8$. Thus $\mathrm{Cl}_2(L_d)\not=(2,2)$. Which achieves the proof.
 	\end{enumerate}
 \end{proof}

  \begin{prop}\label{4}
  	Let $d=p_1p_2$  where $p_1$ and $p_2$ are two primes such that $p_1\equiv p_2\equiv 5 \pmod 8$. Then $h_2(L_d)\equiv 0\pmod 8$ and  $\mathrm{Cl}_2(L_d)\not=(2,2)$.
  \end{prop}
  \begin{proof}
  	Consider the following diagram \begin{center}
  		\vspace*{-0.6 cm}
  		\begin{figure}[H]\label{ diagram Ld/Qi}
  			\hspace*{-3cm}
  			\begin{minipage}{5cm}
  				{\footnotesize
  					%\vspace*{3.3cm}
  					\hspace{0.5cm}\begin{tikzpicture} [scale=1.2]
  					% Positionner les noeuds
  					\node (Q)  at (0,0) {$\mathbb Q(i)$};
  					\node (d)  at (-2.5,1) {$K=\mathbb Q(\sqrt 2,i)$};
  					\node (-d)  at (2.5,1) {$K_1=\mathbb Q(\sqrt{2d},i)$};
  					\node (zeta)  at (0,1) {$K_2= \mathbb{Q}(\sqrt d, i)$};
  					\node (zeta d)  at (0,2) {$L_d=\mathbb Q(\sqrt 2,i,\sqrt d)$};
  					\draw (Q) --(d)  node[scale=0.4,midway,below right]{};
  					\draw (Q) --(-d)  node[scale=0.4,midway,below right]{};
  					\draw (Q) --(zeta)  node[scale=0.4,midway,below right]{};
  					\draw (Q) --(zeta)  node[scale=0.4,midway,below right]{};
  					\draw (zeta) --(zeta d)  node[scale=0.4,midway,below right]{};
  					\draw (d) --(zeta d)  node[scale=0.4,midway,below right]{};
  					\draw (-d) --(zeta d)  node[scale=0.4,midway,below right]{};
  					\end{tikzpicture}}
  			\end{minipage}
  			\caption{$L_{p_1p_2}/\mathbb Q(i)$.}
  		\end{figure}
  	\end{center}
  	\vspace{-0.8cm}	By  Kuroda's class number formula (see \cite{lemmermeyer1994kuroda}), we have
  	$$\begin{array}{ll}
  	h_2(L_d)&=\frac{1}{4}Q(L_d/\mathbb{Q}(i))h_2(K_1)h_2(K)h_2(K_2)/h_2(\mathbb{Q}(i))^2\\
  	&=\frac{1}{4}Q(L_d/\mathbb{Q}(i))h_2(K_1)h_2(K_2).
  	\end{array}$$
  	Using Lemma \ref{lm de parry}, we get $8$ and $4$  divide  $h_2(K_1)$ and   $h_2(K_2)$ respectively. Hence, $8$ divides $h_2(L_{d})$ and so $\mathrm{Cl}_2(L_d)$ is not elementary by Theorem \ref{thm 2-rank = 2}.
  \end{proof} 
  
  \begin{prop}\label{5}
  	Let $d=q_1q_2$  where $q_1$ and $q_2$ are two primes such that $q_1\equiv q_2\equiv 3 \pmod 8$. Then $h_2(L_d)\equiv 0\pmod 8$ and  $\mathrm{Cl}_2(L_d)\not=(2,2)$.
  \end{prop}
  \begin{proof}
  	By  Kuroda's class number formula (see   \cite{lemmermeyer1994kuroda}), we have
  	$$\begin{array}{ll}
  	h_2(L_d)=\frac{1}{4}Q(L_d/\mathbb{Q}(i))h_2(k_1)h_2(k_2),
  	\end{array}$$
  	where $k_1=\mathbb Q(\sqrt{q_1q_2}, i)$ and $k_2=\mathbb Q(\sqrt{2q_1q_2}, i)$. Note that $h_2(k_1)$ is divisible by $2$ (see Lemma \ref{lm de parry}). On the other hand    by Lemma \ref{lm de parry} the rank of the $2$-class group of $k_2$ is $2$ and 
  	by \cite{azizi99, azizitaous(24)(222)} its $2$-class group is not of type $(2,2)$ or $(2,4)$, so 
  	  $h_2(k_2)$ is divisible by $16$.
  	   %(see Lemma \ref{lm de parry} and \cite{azizi99, azizitaous(24)(222)}).
  	     Hence  $h_2(L_d)$ is divisible by $8$. So the result.
  \end{proof}

  To continue, we need   the following two lemmas.

  \begin{lm}\label{corr indices du thm pq}
  	Let $d=pq$ with $p\equiv 5\pmod 8$ and  $q\equiv 7\pmod 8$ are  primes.  Then $\{\varepsilon_2, \varepsilon_{pq}, \sqrt{\varepsilon_{pq}\varepsilon_{2pq}}\}$ is a fundamental system of units of both $L_d$ and $L_d^+=\mathbb{Q}(\sqrt{pq}, \sqrt2)$. Moreover,
  	\begin{enumerate}[\rm 1.]
  		\item $E_{L_d}=\langle\zeta_8, \varepsilon_2, \varepsilon_{pq}, \sqrt{\varepsilon_{pq}\varepsilon_{2pq}}\rangle.$
  		\item  $Q_{L_{d}}=1$ and $q(L_d)=4$.
  	\end{enumerate}
  \end{lm}
  \begin{proof}
  	Let $\varepsilon_{2pq}=x+y\sqrt{2pq}$ and $\varepsilon_{pq}=a+b\sqrt{pq}$ be the fundamental units of $\mathbb{Q}(\sqrt{2pq})$ and $\mathbb{Q}(\sqrt{pq})$ respectively. It is well known that  $N(\varepsilon_{2pq})=N(\varepsilon_{pq})=1$. Then $a^2-1=b^2pq$ and $x^2-1=2y^2pq$. Hence the unique prime factorization in $\ZZ$ implies that one of the numbers: $x\pm1$ (resp. $(a\pm1$)),  $p(x\pm1)$ (resp. $p(a\pm1))$ and $2p(x\pm1)$ (resp. $2p(a\pm1)$) is  a square in $\NN$.  We claim that $x+1$, $a+1$, $x-1$ and  $a-1$ are not squares in $\NN$, otherwise we get for
  	$b=b_1b_2$ and $y=y_1y_2$
  	$$\left\{ \begin{array}{ccc}
  	a\pm1&=&b_1^2\\
  	a\mp1&=&pqb_2^2
  	\end{array}\right. \text{ and }
  	\left\{ \begin{array}{ccc}
  	x\pm1&=&y_1^2\\
  	x\mp1&=&2pqy_2^2.
  	\end{array}\right. $$
  	Hence $1=\left(\frac{b_1^2}{p}\right)=\left(\frac{a\pm1}{p}\right)=\left(\frac{a\mp1\pm 2}{p}\right)=\left(\frac{\pm 2}{p}\right)=\left(\frac{ 2}{p}\right)=-1$ and
  	$1=\left(\frac{y_1^2}{p}\right)=\left(\frac{x\pm1}{p}\right)=\left(\frac{x\mp1\pm 2}{p}\right)=\left(\frac{\pm 2}{p}\right)=\left(\frac{ 2}{p}\right)=-1,$
  	which is absurd.
  	Thus  by \cite[Proposition 3.3]{AZT2016} $\{\varepsilon_2, \varepsilon_{pq}, \sqrt{\varepsilon_{pq}\varepsilon_{2pq}}\}$ is the fundamental system of units of both $L_d$ and $L_d^+$.  So $E_{L_d}=\langle\zeta_8, \varepsilon_2, \varepsilon_{pq}, \sqrt{\varepsilon_{pq}\varepsilon_{2pq}}\rangle$ and $Q_{L_{d}}=1$. As  $\prod_{i=1}^{7}E_{k_i}=\langle i, \varepsilon_2, \varepsilon_{pq}, \varepsilon_{2pq}\rangle$, then $q(L_{d})=4$.
  \end{proof}
  \begin{lm}\label{corr indices du thm q1q2}
  	Let $d=q_1q_2$ with  $q_1\equiv 3\pmod 8$ and $q_2\equiv 7\pmod 8$  are primes. Then $\{\varepsilon_2, \varepsilon_{q_1q_2}, \sqrt{\varepsilon_{q_1q_2}\varepsilon_{2q_1q_2}}\}$ is a fundamental system of units of both $L_d$ and $L_d^+=\mathbb{Q}(\sqrt{q_1q_2}, \sqrt2)$. Moreover,
  	\begin{enumerate}[\rm 1.]
  		\item $E_{L_d}=\langle\zeta_8, \varepsilon_2, \varepsilon_{q_1q_2}, \sqrt{\varepsilon_{q_1q_2}\varepsilon_{2q_1q_2}}\rangle.$
  		\item $Q_{L_{d}}=1$ and $q(L_d)=4$.
  	\end{enumerate}
  \end{lm}
  \begin{proof}
  	We proceed as in the proof of Lemma \ref{corr indices du thm pq}.
  \end{proof}

  \begin{prop}\label{prop d=pq  cl=(2,2)}
  	Let $d=pq$  where $p$ and   $q$ are two primes such that $p\equiv 5 \pmod 8$ and $q\equiv 7 \pmod 8$. Then
  	$$Cl_2(L_{d})=(2,2) \text{ if and only if }   \left(\frac{p}{q}\right)=-1,$$
  	otherwise, $h_2(L_{d})\equiv0\pmod{16}$.
  \end{prop}
%  \begin{proof}
%  	Suppose that $\left(\frac{p}{q}\right)=-1$. By Lemma \ref{wada's f.},   we have
%  	$$\begin{array}{ll}
%  	h_2(L_{d})&=\frac{1}{2^5}q(L_{d})h_2(pq) h_2(-pq)h_2(2pq)h_2(-2pq)h_2(2)h_2(-2)h_2(-1)\\
%%  	&=\frac{1}{2^5}q(L_{d})h_2(pq) h_2(-pq)h_2(2pq)h_2(-2pq)\\
%  	&=\frac{1}{2^5}q(L_{d}).2. 2.2.4  \;\;\;\; (\text{see \cite{kaplan76, connor88}})\\
%  	&=q(L_{d}).
%  	\end{array}$$
 % 	Hence, Lemma \ref{corr indices du thm pq} implies  that $h_2(L_{d})=q(L_{d})=4$. For the converse, assume that $\left(\frac{p}{q}\right)=1$. Since $q(L_d)=4$, then  Lemma \ref{wada's f.}  and \cite{kaplan76}  imply %that
 % 	$$\begin{array}{ll}
%  	h_2(L_{d})&=\frac{1}{2^5}q(L_{d})h_2(pq) h_2(-pq)h_2(2pq)h_2(-pq)h_2(2)h_2(-2)h_2(-1)\\
%  	&=\frac{1}{2}h_2(-pq)h_2(-2pq).
%  	\end{array}$$
%  	Since $h_2(-pq)$ (resp. $h_2(-2pq)$) is divisible by $4$ (resp. $8$), then $h_2(L_{d})$ is divisible by $16$. So the result.
%  \end{proof}
  \begin{proof}
  	By Lemma \ref{corr indices du thm pq}  we have $ q(L_{d})=4$.	It follows by Proposition \ref{wada's f.} that we have
  	$$\begin{array}{ll}
  	h_2(L_{d})&=\frac{1}{2^5}q(L_{d})h_2(pq) h_2(-pq)h_2(2pq)h_2(-2pq)h_2(2)h_2(-2)h_2(-1)\\
  	&=\frac{1}{2^3}h_2(pq) h_2(-pq)h_2(2pq)h_2(-2pq).
  	\end{array}$$
  	\begin{enumerate}[\rm $\bullet$]
  		\item Assume that $\left(\frac{p}{q}\right)=-1$, then by   \cite[Corollaries 19.6 and 19.7]{connor88}      $ h_2(pq)= h_2(-pq)=h_2(2pq)=2$ and by \cite[p. 353]{kaplan76} $h_2(-2pq)=4$.
  	Therefore the above equation gives
  	$$\begin{array}{ll}
  	h_2(L_{d})&= \frac{1}{2^5}q(L_{d})h_2(pq) h_2(-pq)h_2(2pq)h_2(-2pq)\\
  	&=\frac{1}{2^3}\cdot 2\cdot 2\cdot 2\cdot 4=4. 
  	\end{array}$$
  	
  \item Assume that $\left(\frac{p}{q}\right)=1$. By  \cite[Corollary 19.7]{connor88}  we have $ h_2(pq)= h_2(2pq)=2$.
  	Then as above we have
  	$$\begin{array}{ll}
  	h_2(L_{d})&= \frac{1}{2}h_2(-pq)h_2(-2pq).
  	\end{array}$$
  	By Proposition $B_{10}'$ of  \cite[p. 353]{kaplan76},   $h_2(-2pq)$ is divisible by $8$. By  \cite[Corollaries 19.6 and 18.4 ]{connor88} $h_2(-pq)$ is divisible by    $4$. Therefore in this case $h_2(L_{d})$
  	is divisible by $16$.
  	  	\end{enumerate}
  	 Hence Theorem  \ref{thm 2-rank = 2} completes the proof.
  \end{proof}

  \begin{prop}\label{prop d=q1q2  cl=(2,2)}
  	Let $d=q_1q_2$   where $q_1$, $q_2$ are two primes such that  $q_1\equiv 3 \pmod 8$ and $q_2\equiv 7 \pmod 8$. Then
  	$$\mathrm{Cl}_2(L_{d})=(2, 2) \text{ if and only if }\left(\frac{q_1}{q_2}\right)=-1,$$
  	otherwise, $h_2(L_{d})$ is divisible by $16$.
  \end{prop}
  \begin{proof}
  	The proof is similar to the one of Proposition \ref{prop d=pq  cl=(2,2)}.
  \end{proof}

  % \noindent Here are some  numerical examples illustrating our results.
  \begin{exams}The examples are given by using PARI/GP software and they confirm our results.
  	\begin{enumerate}[\rm 1.]
  		\item We have $\left(\frac{2}{17}\right)_4=-\left(\frac{17}{2}\right)_4=-1$ and the 2-class group of $\QQ(\sqrt {17}, \sqrt2, i)$ is of type $(2,2)$.
  		\item  The $2$-class groups of  the fields $\QQ(\sqrt {21}, \sqrt2, i)$ and $\QQ(\sqrt {35}, \sqrt2, i)$ are of type $(2, 2)$.
  		\item The $2$-class group of  the field $\QQ(\sqrt {19\cdot23}, \sqrt2, i)$ is of type $(2, 2)$. Whereas, that of $\QQ(\sqrt {19\cdot31}, \sqrt2, i)$ is not, since the  $2$-part of its class number   is divisible by $16$. In fact $\left(\frac{19}{31}\right)=-\left(\frac{19}{23}\right)=1$.
  	\end{enumerate}
  \end{exams}

  \subsection{The $2$-part of the class number  of some biquadratic fields}\label{sous-section sur 2-partie des corps biquadratiques}
  
  In this subsection we  use the previous results to %compute the $2$-class number of some biquadratic fields. Furthermore,  we 
  give the $2$-class number of $L_d$ in terms  of that of  $\mathbb{Q}(\sqrt{2},\sqrt{-d})$.

  \begin{theorem}\label{thm 4.17}
  	Let $d=pq$ where $p$, $q$ are two primes such that  $p\equiv 5 \pmod 8$  and $q\equiv 7 \pmod 8$. Then
  	$$h_2(L_d)=h_2(k),$$
  	with $k=\mathbb{Q}(\sqrt{2},\sqrt{-d})$. Moreover,
  	\begin{eqnarray*}
  		\mathrm{Cl}_2(k)=(2,2) \text{ if and only if} \left(\frac{p}{q}\right)=-1.
  	\end{eqnarray*}
  \end{theorem}	
  \begin{proof}
  	Applying Lemma \ref{lm Kurodas formula} to the extension   $L_{d}/\mathbb{Q}(\sqrt{2})$, we get
  	\begin{eqnarray*}
  		h_2(L_d)=\frac{Q_{L_d}}{Q_{k}Q_{K}}\frac{1}{2}
  		h_2( L_{d}^+)h_2(k).
  	\end{eqnarray*}
  	By  Proposition \ref{wada's f.},   Lemma \ref{corr indices du thm pq} and the settings on values of class numbers of quadratic fields given in the proof of Proposition \ref{prop d=pq  cl=(2,2)} we obtain 
  	$$h_2( L_{d}^+)=\frac{1}{4}q(L_{d}^+)h_2(pq)h_2(2pq)h_2(2)=\frac{1}{4}\cdot 2\cdot 2\cdot 2\cdot1=2.$$ 
  	By Lemmas \ref{lm unite zeta8} and   \ref{corr indices du thm pq} (resp.  \cite[p. 19]{AZ99}), we have $Q_{L_d}=Q_{K}=1$ (resp.  $Q_{k}=1$). Thus
  	$h_2(L_d)=	h_2(k).$
  	Since by \cite[Proposition 2]{mccall1995imaginary}, the rank of the $2$-class group of $k$ is $2$, we have the equivalence by  Proposition \ref{prop d=pq  cl=(2,2)}.
  \end{proof}

%  \begin{rema}
%  	With the same assumptions as in the previous theorem, we have
%  	\begin{enumerate}[\rm 1.]
%  		\item $	h_2(k)=4$ if and only if $\left(\frac{p}{q}\right)=-1.$
 % 		\item $h_2(L_d^+)=2$ with $L_d^+=\mathbb{Q}(\sqrt{2},\sqrt{d}).$
%  	\end{enumerate}
 % \end{rema}

  \begin{theorem}
  	Let $d=q_1q_2$ where $q_1$, $q_2$ are two primes such that  $q_1\equiv 3 \pmod 8$  and $q_2\equiv 7 \pmod 8$. Then
  	$$h_2(L_d)=\frac{1}{2}h_2(k),$$
  	with $k=\mathbb{Q}(\sqrt{2},\sqrt{-d})$. Moreover,
  	$$\mathrm{Cl}_2(k)=(2,2,2) \text{ if and only if} \left(\frac{q_1}{q_2}\right)=-1.$$
  \end{theorem}
  \begin{proof}
  	We proceed as in the   proof of Theorem \ref{thm 4.17}.
  \end{proof}
  
 % \begin{rema}
 % 	With the same assumptions of the previous theorem, we have
%  	\begin{enumerate}[\rm 1.]
 % 		\item $	h_2(k)=8$ if and only if $\left(\frac{q_1}{q_2}\right)=-1.$
 % 		\item $h_2(L_d^+)=1$ with $L_d^+=\mathbb{Q}(\sqrt{2},\sqrt{d}).$
 % 	\end{enumerate}
 % \end{rema}

  \begin{rema}
  	 As a continuation of this work  we interested in the cases where  $\mathrm{Cl}_2(L_d)$ is of type $(2^n, 2^m)$ or $(2,2,2)$,  where $n\geq1$ and $m>1$, and based on this work, we studied the problem of the Hilbert $2$-class field tower of these fields. Note also that,    we generalized some of our results on $\mathbb{Q}(\zeta_{8},\sqrt{d})$ to the fields  $\mathbb{Q}(\zeta_{2^m},\sqrt{d})$, for $m\geq 4$ (cf. \cite{chemsZkhnin4,chemsZkhnin2,CZA3}).  
  \end{rema}
  
 % \section{Appendix}\label{3}
 % Let $d$ be an odd positive square-free integer. We have $L_d=\mathbb{Q}(\sqrt{i},\sqrt{d})=\mathbb{Q}(i,\sqrt{2},\sqrt{d})$. Using the primitive element theorem,  we get
 % $L_{d}=\mathbb{Q}(\alpha)$ with $\alpha =\sqrt{d}+\sqrt{i}$. One can   easily verify that
 % $$P_d(X)=X^{8}-4dX^{6}+(6d^2+2)X^4+(-4d^3+12d)X^2+(d^4+2d^2+1),$$
 % is a  polynomial of degree $8$ that vanishes at $\alpha$. So it is irreducible. Since $e^{i\pi/4}$ and $-e^{i\pi/4}$ are two square roots of $i$, then $\alpha_1=\sqrt{d}+e^{i\pi/4}$ and $\alpha_2=\sqrt{d}-e^{i\pi/4}$ are two roots of $P_d$. Remark that $P_d(X)=Q_d(X^2)$ and $P_d$ is a polynomial  with real coefficients, then   the roots of $P_d$ are listed as follows:\\
  
 % \hspace{ 0.3cm}{\renewcommand{\arraystretch}{1.2}
 % 	\setlength{\tabcolsep}{0.4cm}
 % 	\begin{tabular}{|c|c|c|c|}
  %		\hline
  %		$x$	& $\overline{x}$ & $-x$ & $\overline{-x}$ \\
  %		\hline
  %		$\alpha_1=\sqrt{d}+e^{i\pi /4}$	& $\sqrt{d}+e^{-i\pi /4}$ &  $-\sqrt{d}-e^{i\pi /4}$ &  $-\sqrt{d}-e^{-i\pi /4}$ \\
  %		\hline
  %		$\alpha_2=\sqrt{d}-e^{i\pi /4}$	& $\sqrt{d}-e^{-i\pi /4}$ &$-\sqrt{d}+e^{i\pi /4}$  & $-\sqrt{d}+e^{-i\pi /4}$ \\
  %		\hline
 % \end{tabular} }	

   \section*{Acknowledgment}
  	The authors are very grateful to the reviewer for his/her careful and meticulous reading of the paper.


\begin{thebibliography}{11}
  	
  	\bibitem{azizi99} A. Azizi,  {\it Sur le $2$-groupe de classes d'idéaux   de $\QQ(\sqrt d, i)$,}{ Rend. Circ. Mat. Palermo,  (2) 48 (1999), 71–92.}
  	
  	\bibitem{AZ99} A. Azizi,  {\it Unités de certains corps de nombres imaginaires et abéliens sur $\mathbb{Q}$,}{ Ann. Sci. Math. Québec, 23 (1999), 15–21.}
  	
  	 % \bibitem{azizi99unite}   A. Azizi,   {\it Unités de certains corps de nombres imaginaires et abéliens sur $\mathbb{Q}$,}{ Ann. Sci. Math. Québec, {\bf23} (1999),  15-21.}
  	
  	
  	
  	\bibitem{Be05} A. Azizi et I. Benhamza,  {\it Sur la capitulation des $2$-classes d'idéaux de $\QQ(\sqrt d, \sqrt{-2})$,}{ Ann. Sci. Math. Québec., { 29} (2005), 1-20.}
  	
  	
  	
  	\bibitem{azizi20042}
  	A.~Azizi et A.~Mouhib, {\it Le $2$-rang du groupe de classes de certains corps
  		biquadratiques et applications,}{ Internat. J. Math.,
  		15 (2004),  169--182.}
  	
  	\bibitem{AM} A. Azizi  et A. Mouhib, {\it Sur le rang du $2$-groupe de classes de $\mathbb{Q}( \sqrt{m},\sqrt{d} )$ où $m=2$ ou un premier $p \equiv 1 \pmod{4}$ .} {Trans. Amer. Math. Soc., 353 (2001),  2741–2752.}
  	
  	\bibitem{taous2008}
  	A.~Azizi et M.~Taous, {\it Capitulation des $2$-classes d'id\'eaux de
  		$k=\mathbb{Q}(\sqrt{2p},i)$,}{  Acta Arith., 131 (2008),  103–123.}
  	
  	\bibitem{azizitaous(24)(222)}
  	A.~Azizi and M.~Taous, {\it Déterminations des corps
  		$k=\mathbb{Q}(\sqrt{d},\sqrt{-1})$ dont les $2$-groupes de classes sont de type
  		$(2, 4)$ ou $(2, 2, 2)$,}{ Rend. Istit. Mat. Univ. Trieste, 40 (2008), 93–116.}
  	
  	%\bibitem{AZ2018}
  	%A. Azizi and A. Zekhnini, {\it The structure of the second $2$-class group of some special Dirichlet fields,}{ To apear in Asian-Eur. J. Math., doi.org/10.1142/S1793557120500539.}
  	
  	\bibitem{AZT2016}
  	A. Azizi, A. Zekhnini and M. Taous, {\it On the strongly ambiguous classes of some biquadratic number  fields, }{Math. Bohem., 141 (2016), 363–384.}
  	
  	
  	
  	
  	
  	
  %	\bibitem{cohn}
  %	P.~Barrucand and H.~Cohn, {\it Note on primes of type $x^2 + 32y^2$, class number, and
  	%	residuacity,}{ J. Reine. Angew. Math., 238 (1969), 67-70.}
  	
  	\bibitem{Parry77}
  	E. Brown and C. J. Parry. {\it The $2$-class group of certain biquadratic number fields,}{ J.  reine angew.  Math.,   295 (1977), 61-71.}
  	
  	\bibitem{ezrabrown} E. Brown,  \textit{ The class number of $\mathbb{Q}(\sqrt{-p})$  for $p\equiv 1\pmod 8$ a prime,}{ Proc. Amer. Math. Soc., 31 (1972),  381--383.}
  	
  	
  	\bibitem{brown19782}
  	E. Brown and C. J. Parry. {\it The $2$-class group of biquadratic number fields II,}{ Pacific J.  Math.,    78 (1978),   11-26.}
  	
  	
  	
  	%C
  	%\bibitem{chevalley1933theorie} C.~Chevalley, {\it Sur la theorie du corps de classes dans les corps finis et les corps locaux,}{ J. Fac. Sc. Tokyo, Sect. 1, t.2, (1933), 365-476.}
  	
  	\bibitem{chemsZkhnin4}M. M. Chems-Eddin, A. Azizi  and A. Zekhnini,  {\it  On an infinite family of imaginary triquadratic number fields}, Stud. Fuzziness Soft Comput., 395 (2021),   211-215.
  	\bibitem{chemsZkhnin2}M. M. Chems-Eddin, A. Azizi  and A. Zekhnini,   {\it On the $2$-class group of some number fields with large degree}.	arXiv:1911.11198.
  	
  	\bibitem{CZA3}M. M. Chems-Eddin,   A. Zekhnini and A. Azizi, \textit{	Units and  $2$-class field towers  of some multiquadratic number fields},   Turk. J. Math, 44, (2020), 1466-1483.
  	
  	
  	\bibitem{connor88}
  	P. E. Conner and J. Hurrelbrink, {\it Class number parity,}{ Ser. Pure Math., vol. 8. World Scientific, Singapore (1988).}
  	
  	%F
  	\bibitem{frohlich1993}
  	A.~Fr{\"o}hlich  and M.~J. Taylor, {\it  Algebraic number theory,}{ Cambridge Studies in Advanced Mathematics, 27. Cambridge University Press, Cambridge, 1993.}
  	
  	%\bibitem{gras2013class}G.~Gras, {\it Class Field Theory: from theory to practice,}{ Translated from the French manuscript by Henri Cohen. Springer Monographs in Mathematics. Springer-Verlag, Berlin, 2003. xiv+491 pp. ISBN: 3-540-44133-6.}
  	
  	\bibitem{gras1973classes} G.~Gras, {\it Sur les $l$-classes d'idéaux dans les extensions cycliques relatives de degré premier $l$,}{ Ann. Inst. Fourier (Grenoble); 23 (1973), 1–48.}
  	  	\bibitem{grasbook} G.~Gras, {\it     Class Field Theory: From Theory to Practice}, Springer-Verlag Berlin Heidelberg, 2003.
  	
  	
  	%\bibitem{H. Hasse} H.~Hasse, {\it Neue begr{\"u}ndung und verallgemeinerung der Theorie des normenrestsymbols, }{J. Reine Angew. Math., 162 (1930), 134-144.}
  	%\bibitem{hasse} H. Hasse, \it{Neue Begr\"undung der theorie der Normenrest symbols}, J. Reine Angew. Math. 162(1930).
  	
  	%\bibitem{kaplan1973divisibilitepar8} P.~Kaplan, {\it Divisibilité par $8$ du nombre de classes des corps quadratiques dont le $2$-groupe des classes est cyclique et réciprocité biquadratiques,}{ J. Math. Soc. Japan 25 (1973), no 4, 596–608.}
  %	\bibitem{He-36}
  %	J. Herbrand, {\it Le développement moderne de la théorie des corps algébriques: corps de classes et lois de reciprocite,}{ Mém. Sci. Math., Fasc. LXXV, Gauthier–Villars, Paris 1936.}
  	
  	\bibitem{kaplan76}
  	P.~Kaplan, {\it Sur le $2$-groupe de classes d'idéaux des corps quadratiques,}{ J. Reine
  		angew. Math., 283/284 (1976), 313-363.}
  	
  	\bibitem{kuvcera1995parity}
  	R.~Ku{\v{c}}era, {\it On the parity of the class number of a biquadratic field,}
  	{   J. Number Theory, 52 (1995),  43–52.}
  	
  	%L
  	\bibitem{lemmermeyer1995ideal}
  	F.~Lemmermeyer, {\it Ideal class groups of cyclotomic number fields I,}{  Acta Arith., 72 (1995),  347–359.}
  	
  	
  	\bibitem{lemmermeyer1994kuroda}
  	F.~Lemmermeyer, {\it Kuroda's class number formula.}{  Acta Arith. 66 (1994),  245–260.}
  	
  	\bibitem{lemmermeyer2013reciprocity}F.~Lemmermeyer, {\it Reciprocity laws. From Euler to Eisenstein, }{ Springer Monographs in Mathematics. Springer-Verlag, Berlin, 2000.}
  	
  	%\bibitem{leonard1982divisibility} P.~A. Leonard and K.~S. Williams, {\it On the divisibility of the class numbers of	$\mathbb{Q}(\sqrt{-d})$ and $\mathbb{Q}(\sqrt{-2d})$ by $16$, }{Can. Math. Bull. 25 (1982), 200-206.}
  	
  	%\bibitem{masley1976}J.~M. Masley and H.~L. Montgomery, {\it Cyclotomic fields with unique	factorization,}{ J. reine angew. Math. vol.~286/287, (1976),~248--256.}
  	
  	%\bibitem{mccall19972} T.~M. McCALL, C.~J. Parry, and R.~R. Ranalli, {\it The $2$-rank of the class group	of imaginary bicyclic biquadratic fields,}{ Canad. J. Math. vol.~49, no.~2, (1997), 283--300.}
  	
  	\bibitem{mccall1995imaginary} T. M. McCall, C. J. Parry and R. R. Ranalli, {\it Imaginary bicyclic biquadratic fields with cyclic $2$-class group,}{ J. Number Theory, 53 (1995),  88–99.}
  	
  	
  	\bibitem{Scholz1935}
  	A.~Scholz, {\it {\"U}ber die l{\"o}sbarkeit der gleichung $t^2-Du^2=-4$,}{ Math. Z. 39 (1935),  95–111.}
  	
  	
  	\bibitem{wada}H. Wada. {\it On the class number and the unit group of certain algebraic number fields,}{ J. Fac. Sci. Univ. Tokyo, 13 (1966),  201--209.}
  	
  	\bibitem{washington1997introduction}
  	L.~C. Washington, {\it Introduction to cyclotomic fields.}{  Graduate Texts in Mathematics, 83. Springer,  1997.}
  	
  	
  	
  \end{thebibliography}
 \end{document}